\RequirePackage{fix-cm,amsmath}
\documentclass[smallextended]{svjour3}       % onecolumn (second format)
\smartqed  % flush right qed marks, e.g. at end of proof
\usepackage{amssymb}
\usepackage{amscd}
\usepackage{soul}
\usepackage{cancel}
\usepackage{graphicx}

\usepackage{units}
\usepackage{enumitem}
\setlist[enumerate,1]{label=(\alph*), ref=(\alph*)}
\usepackage{hyperref}
\usepackage{subfloat}
\usepackage{subcaption}
\usepackage{geometry}
\usepackage{authblk}
\usepackage{xargs}
\usepackage[linesnumbered]{algorithm2e}
\usepackage[most]{tcolorbox}
\usepackage{mathtools,halloweenmath}
\usepackage[dvipsnames]{xcolor}

\usepackage[normalem]{ulem}
\newcommand{\soutb}[1]{\bgroup\markoverwith
  {\textcolor{Brown}{\rule[0.5ex]{2pt}{1pt}}}\ULon{#1}\egroup}

\usepackage{mathtools}

\tcbset{colback=yellow!10!white, colframe=black, 
        highlight math style= {enhanced, %<-- needed for the ’remember’ options
            colframe=red,colback=red!10!white,boxsep=0pt}
       }
 
 \DeclareMathOperator*{\argmin}{arg\,min}
 \newcommand\innprod[2]{\left\langle{}#1{},{}#2{}\right\rangle}%
 \newcommand\func[3]{#1:#2\rightarrow#3}%	     	f: A --> B function \func{f}{A}{B}
 \newcommand\R{\mathbb R}%
 \newcommand\N{\mathbb N}%
 \newcommand\Ni{\mathcal{N}_i}%
 \newcommand\C{\mathcal C}%
 \newcommand\X{\mathcal X}%
 \newcommand\K{\mathcal K}%
 \newcommand\Rinf{\overline\R}
 \newcommand\ov[1]{\overline #1}%
 \newcommand\dom{\mathrm{dom}}%
 \newcommand\interior{\mathrm{int}}%
 \newcommand\idom{\mathrm{int\, dom}}%
 \newcommand\dist{\mathrm{dist}}%
 \newcommandx\seq[3][{2=k\in \N_{0}},{3={}}]{\{#1\}_{#2}^{#3}}% sequence \seq{x^k}[k=1][n]
 \newcommandx\subseq[3][{2=j\in \N_{0}},{3={}}]{\{#1\}_{#2}^{#3}}% sequence \seq{x^k}[k=1][n]

 \newtheorem{thm}{Theorem}
 \numberwithin{thm}{section}
 \newtheorem{lem}[thm]{Lemma}
 \newtheorem{prop}[thm]{Proposition}
 \newtheorem{cor}[thm]{Corollary}
 \newtheorem{rem}[thm]{Remark}
 \newtheorem{exa}[thm]{Example}
 \newtheorem{defin}[thm]{Definition}
 \newtheorem{fact}[thm]{Fact}
 \newtheorem{ass}{Assumption}

 \journalname{}
\begin{document}

\title{Minimizing Smooth Kurdyka-{\L}ojasiewicz Functions via Generalized Descent Methods: Convergence Rate and Complexity
}

\titlerunning{Minimizing Smooth Kurdyka-{\L}ojasiewicz Functions via Generalized Descent Methods}        % if too long for running head

\author{Masoud Ahookhosh \and Susan Ghaderi \and Alireza Kabgani \and Morteza Rahimi
}


\institute{M. Ahookhosh, A. Kabgani, M. Rahimi \at
              Department of Mathematics, University of Antwerp, Middelheimlaan 1, B-2020 Antwerp, Belgium. \\
              Tel.: +123-45-678910\\
              \email{masoud.ahookhosh@uantwerp.be, alireza.kabgani@uantwerp.be, morteza.rahimi@uantwerp.be}            \\
              \and 
              S. Ghaderi,\at
              Department of Electrical Engineering (ESAT), KU Leuven, Leuven, Belgium\\
              \email{susan.ghaderi@kuleuven.be}\\
              The first author was partially supported by the Research Foundation Flanders (FWO) research project G081222N and UA BOF DocPRO4 projects with ID 46929 and 48996. SG also acknowledges the support of the Research Foundation Flanders (FWO) research project 12AK924N.
}

\date{Received: date / Accepted: date}

\maketitle
\vspace{-5mm}
\begin{abstract}
This paper introduces a generalized descent algorithm (DEAL) for minimizing smooth nonconvex functions. If the objective function is nonsmooth, a smoothing technique (e.g., forward-backward and high-order Moreau envelopes) is applied to generate a smooth counterpart. The proposed framework unifies several methods, such as gradient-based methods with constant step-sizes and Armijo line search, and several proximal splitting methods. The method is built around a generalized descent inequality that adapts the amount of decrease to the geometry of the objective function. Under the Kurdyka-{\L}ojasiewicz (KL) property, we establish global convergence of the generated sequence to critical points and provide a unified convergence rate analysis. In particular, we show that the convergence behavior depends jointly on the KL exponent and the descent order, and we identify a precise condition under which generalized descent methods achieve linear convergence. By choosing the order of high-order proximal regularization according to the KL exponent, our boosted high-order proximal-point method achieves linear convergence for arbitrary KL exponents.
If the objective function satisfies a global KL inequality, we further strengthen the results by proving convergence to global minimizers and deriving explicit iteration-complexity bounds. Numerical experiments validate our theoretical foundation.

\keywords{Nonsmooth and nonconvex optimization \and Smoothing techniques \and Generalized descent methods \and Kurdyka-{\L}ojasiewicz functions \and Convergence rate \and Complexity analysis.}
\end{abstract}
${}$\\
Communicated by Felipe Lara Obreque.
\vspace{-4mm}
%%%%%%%%%%%%%%%%%%%%%%%%%%%%%%%%%%%%%%%%%%%%%%%%%%%%%%%%%%%%%%%%%%%%%%%%%%%%%%%%%%
%%%%%%%%%%%%%%%%%%%%%%%%%%%%%%%%%%%%%%%%%%%%%%%%%%%%%%%%%%%%%%%%%%%%%%%%%%%%%%%%%%
\section{Introduction} \label{sec:introduction}
Let us consider the unconstrained optimization problem
\begin{equation}\label{eq:prob}
    \min_{x\in\R^n}~f(x),
\end{equation}
under the following standing assumption:
\begin{ass}[Basic assumptions]\label{ass:basic} For problem \eqref{eq:prob}, we assume:
\begin{description}
\item[(a)] The function $\func{f}{\R^n}{\R}$ is continuously differentiable, i.e., $f\in\C^1$, and possibly nonconvex;
\item[(b)] The set of minimizers $\mathcal{X}^*:=\argmin_{x\in \R^n} f(x)$ is nonempty, with $f^*$ denoting the optimal value.
\end{description}
\end{ass}
Despite the lack of convexity, a wide class of descent algorithms is known to exhibit favorable convergence behavior in practice and often converge to \textit{global minimizers}; see, e.g., \cite{attouch2013convergence,karimi2016linear}. 
Nevertheless, the theoretical foundations of such methods remain incomplete, particularly regarding \textit{global convergence} guarantees, \textit{convergence rates}, and, most notably, \textit{iteration complexity} under nonconvexity.

A seminal contribution in this direction was made by Attouch, Bolte, and Svaiter \cite{attouch2013convergence}, who established abstract convergence of descent methods to critical points under two general conditions: a \textit{sufficient decrease} condition and a \textit{relative error} condition. Subsequent works, including \cite{bento2025convergence,frankel2015splitting}, refined this framework and obtained linear convergence rates in certain settings. These analyses typically rely on descent inequalities of the form
\begin{equation}\label{eq:sufdeccond}
f(x^{k+1})+a\Vert x^{k+1}-x^k\Vert ^2\leq f(x^k), \quad a>0.
\end{equation}

In classical smooth optimization methods such as steepest descent, conjugate gradient, Newton, and quasi-Newton schemes, inequality \eqref{eq:sufdeccond} follows naturally from updates of the form $x^{k+1}=x^k + \alpha \ov d^k$, where $\alpha>0$ and $\ov d^k$ is a \textit{sufficient descent direction} \cite{nocedal2006numerical}. Recall that a direction $\ov d^k$ is a {\it descent direction} at $x^k$ if
\begin{equation*}\label{eq:desDir}
    f'(x^k;\ov d^k):=\mathop{\lim}\limits_{t\downarrow 0}\tfrac{f(x^{k}+t \ov d^k)-f(x^{k})}{t}=\big\langle\nabla f(x^k),\ov d^k\big\rangle<0,
\end{equation*}
and a \textit{sufficient descent direction} if there exist constants $c_1,c_2>0$ such that
\begin{equation}\label{eq:sufDescCond}
    \big\langle\nabla f(x^k),\ov d^k\big\rangle\leq -c_1\|\nabla f(x^k)\|^2, \quad \|\ov d^k\|\leq c_2 \|\nabla f(x^k)\|.
\end{equation}
A combination of \eqref{eq:sufdeccond} and \eqref{eq:sufDescCond} yields the classical descent inequality
\begin{equation}\label{eq:generalstr:2}
f(x^{k+1})\leq f(x^k)- \rho\Vert \nabla f(x^k)\Vert ^2,\quad \rho>0.
\end{equation}
While inequality \eqref{eq:generalstr:2} covers a broad family of descent schemes, recent developments 
indicate that many relevant algorithms exhibit non-quadratic decrease behavior. In particular, 
H\"olderian gradient methods \cite{yashtini2016global} and high-order proximal and smoothing approaches \cite{Kabganitechadaptive,Kabganidiff} naturally lead to more general descent relations of the form
\begin{equation}\label{eq:generalstr:p}
f(x^{k+1})\leq f(x^k)- \rho\Vert \nabla f(x^k)\Vert ^\theta,\quad \rho>0, ~~\theta>1.
\end{equation}
Such relations may hold even when the underlying search directions do not satisfy the classical sufficient descent condition \eqref{eq:sufDescCond}. 
Existing convergence frameworks do not fully accommodate this generalized descent structure, thereby motivating the introduction of \textit{generalized (sufficient) descent directions} (Definition~\ref{def:G-descent-dir})
and the development of a unified analytical framework capable of handling
\eqref{eq:generalstr:p}

From a theoretical standpoint, analyzing descent mechanisms of the form
\eqref{eq:generalstr:p} requires a geometric condition that directly links
function values to gradient norms and remains robust under non-quadratic
decrease regimes. The Kurdyka-{\L}ojasiewicz  (KL) framework provides
a natural and sufficiently general tool for studying the convergence behavior
of generalized descent methods beyond the classical quadratic decrease regime. In particular, for methods satisfying \eqref{eq:generalstr:2}, linear convergence is obtained when the KL exponent satisfies $\vartheta=\tfrac{1}{2}$.
A central result of this paper shows that, under the generalized descent condition \eqref{eq:generalstr:p}, linear convergence holds whenever $\vartheta=\tfrac{1}{\theta}$.
 This characterization establishes a precise relationship between the descent order and the KL geometry of the objective function. It further implies that, by appropriately selecting the descent order $\theta$, linear convergence can be achieved for \textit{any} KL exponent $\vartheta\in(0,1)$. High-order proximal methods \cite{ahookhosh2025asymptotic,Kabganitechadaptive,Kabganidiff}, in which the regularization involves a power $p>1$, naturally fall within this setting, enabling linear convergence across the full spectrum of KL exponents.

Although several works \cite{attouch2013convergence,bento2025convergence,frankel2015splitting,Kabganitechadaptive,Kabganidiff} have established convergence and, in some cases, linear rates under KL-type assumptions, iteration complexity has received comparatively little attention. In this paper, we address this gap by deriving explicit complexity bounds under the global KL inequality.
 This assumption, which holds for broad classes of functions such as semialgebraic and globally subanalytic ones \cite{attouch2010proximal,bolte2006nonsmooth,bolte2007lojasiewicz,bolte2007clarke}, guarantees global and linear convergence to optimal solutions and enables precise complexity analysis. It also arises naturally in various applications including reinforcement learning \cite{agarwal2021theory,mei2020global,yuan2022general}, neural networks \cite{allen2019convergence,zeng2018global}, optimal control \cite{bu2019lqr,fatkhullin2021optimizing}, and stochastic optimization \cite{fatkhullin2022sharp}.

The goal of this work is to introduce a \textit{generalized descent algorithm} (DEAL) that unifies a variety of descent schemes, including the constant step-size descent methods, Armijo line search, proximal gradient algorithms, and high-order proximal-point methods within a single theoretical framework. 
We establish global convergence and linear rates, and explicit iteration complexity bounds, thereby providing a comprehensive theoretical foundation for generalized descent methods under KL geometry.
Numerical experiments are provided on representative problems to illustrate the theoretical findings.

\vspace{-5mm}
%%%%%%%%%%%%%%%%%%%%%%%%%%%%%%%%%%%%%%%%%%%%%%%%%%%%%%%%%%%%%%%%%%%%%%%%%%%%%%%%%%
\subsection{{\bf Contribution}} \label{sec:contribution}
Our main contributions are summarized as follows:
\begin{description}
  \item[{\bf (i)}] {\bf Generalized descent framework.}
We introduce a generalized descent framework, named \emph{DEAL}, generating the sequences $\seq{x^k}$ satisfying the inequality \eqref{eq:generalstr:p}. This framework encompasses a broad class of classical and modern algorithms, including but not limited to the constant step-size gradient descent for $f\in\C^{1,\nu}_L$ with $\nu\in(0,1]$, Armijo-type line search methods, proximal-gradient schemes, and (boosted) high-order proximal-point algorithms. In doing so, we extend the classical notion of sufficient descent directions to the more flexible concept of \textit{generalized sufficient descent directions}, which enables algorithmic design beyond the quadratic descent regime.

\item[{\bf (ii)}] {\bf Convergence analysis under KL property.}
Under the Kurdyka-{\L}ojasiewicz (KL) property of the objective function, we establish global convergence of the DEAL framework and provide a unified convergence rate analysis. We show that the rate depends jointly on the KL exponent $\vartheta$ and the descent order $\theta$ given in \eqref{eq:generalstr:p}, and we establish linear convergence whenever $\vartheta=\tfrac{1}{\theta}$. This reveals a precise geometric relationship between the objective landscape and the design of descent mechanisms. As a notable consequence, we show that boosted high-order proximal-point methods achieve \textit{global linear convergence for arbitrary KL exponents}.

\item[{\bf (iii)}] {\bf Complexity under global KL property.}
By imposing the global KL property, we strengthen convergence guarantees to global optimality and derive explicit iteration complexity bounds. These bounds quantify the number of iterations required to reach prescribed accuracies in function values, gradient norms, and iterates and apply uniformly across all algorithmic instances covered by the DEAL framework. This provides, to the best of our knowledge, one of the first comprehensive complexity analyses for generalized descent methods under KL geometry.

\end{description}

\vspace{-5mm}
%%%%%%%%%%%%%%%%%%%%%%%%%%%%%%%%%%%%%%%%%%%%%%%%%%%%%%%%%%%%%%%%%%%%%%%%%%%%%%%%%%
\subsection{{\bf Organization}} \label{sec:organization}
\vspace{-1mm}
The paper is organized as follows. 
Section~\ref{sec:prelim} introduces the notation and preliminary results. 
Section~\ref{sec:genericDescMethod} presents the generalized descent framework and establishes convergence, rate, and complexity results under KL assumptions.
Section~\ref{sec:InstancesDEAL} discusses generalized descent directions and shows how several classical and modern algorithms arise as instances of the proposed framework.
Preliminary numerical experiments are reported in Section~\ref{sec:preNumExper}, and concluding remarks are given in Section~\ref{sec:conclusion}.

%%%%%%%%%%%%%%%%%%%%%%%%%%%%%%%%%%%%%%%%%%%%%%%%%%%%%%%%%%%%%%%%%%%%%%%%%%%%%%%%%%
%%%%%%%%%%%%%%%%%%%%%%%%%%%%%%%%%%%%%%%%%%%%%%%%%%%%%%%%%%%%%%%%%%%%%%%%%%%%%%%%%%
\section{Preliminaries} \label{sec:prelim}

In this paper, $\R^n$ denotes the $n$-dimensional real \textit{Euclidean space} equipped with the standard \textit{inner product} $\innprod{\cdot}{\cdot}$ and the associated \textit{Euclidean norm} $\|\cdot\|=\sqrt{\innprod{\cdot}{\cdot}}$.
The notation $\N$ stands for the set of \textit{natural numbers}, and $\N_{0}:=\N\cup \{0\}$.
The \textit{open ball} of radius $r>0$ centered at $x\in\R^n$ is denoted by $B(x;r)$.
The \textit{interior} of a set $S\subseteq\R^n$ is denoted by $\interior\, S$.
The \textit{Euclidean distance} from a point $x\in\R^n$ to a nonempty set $S\subseteq\R^n$ is given by $\dist(x;S)=\inf_{z\in S} \|z-x\|$.
The \textit{ceiling function}, denoted by $\lceil x \rceil$, assigns to each real number $x$ the smallest integer greater than or equal to it.
A point $\ov{x}$ is called a \textit{cluster point} of a sequence $\{x^k\}_{k \in \N_0}$ if there exists a subsequence converging to $\ov{x}$.
%%%%%%%%%%%%%%%%%%    effective domain:%%%%%%%%%%%%%%%%%%%%%%%%%%
For a function $h: \R^n \to \Rinf:=\R\cup\{+\infty\}$, the \textit{effective domain} is defined as 
$\dom{(h)}:= \{x \in   \R^{n}\mid~h(x)< + \infty \}$,
and $h$ is called \textit{proper} if $\dom(h)\neq \emptyset$.
%%%%%%%%%%%%%%%%%%%%%%%%%%
A function $h$ is \textit{lower semicontinuous} (lsc) at $\ov{x} \in \R^n$ if, for any sequence 
$\{x^k\}_{k\in \mathbb{N}} \subseteq \R^n$ with $x^k \to \ov{x}$, it follows that
$\liminf_{k \to +\infty} h(x^k) \geq h(\ov{x})$.

A proper function $h: \R^n \to \Rinf$ is said to be \textit{Fr\'{e}chet differentiable} at $\ov{x}\in \idom(h)$ with \textit{Fr\'{e}chet derivative} $\nabla h(\ov{x})$ if
\[
\lim_{x\to \ov{x}}\tfrac{h(x) -h(\ov{x}) - \langle \nabla h(\ov{x}) , x - \ov{x}\rangle}{\Vert x - \ov{x}\Vert}=0.
\]
The point $\ov{x}\in \idom(h)$ is called a \textit{critical point} of $h$ if $\nabla h(\ov x)=0$.

%%%%%%%%%%%%%%%%%%%%%%%%%%
The \textit{Fr\'{e}chet/regular} and \textit{Mordukhovich/limiting subdifferentials} of $h$ at $\ov{x}\in \dom{(h)}$ \cite{Mordukhovich2018,rockafellar2011variational} are defined, respectively, as
\[
\widehat{\partial}h(\ov{x}):=\left\{\zeta\in \R^n\mid~\liminf_{x\to \ov{x}}\tfrac{h(x)- h(\ov{x}) - \langle \zeta, x - \ov{x}\rangle}{\Vert x - \ov{x}\Vert}\geq 0\right\},
\]
and
\[
\partial h(\ov{x}):=\left\{\zeta\in \R^n\mid~\exists x^k\to \ov{x}, \zeta^k\in \widehat{\partial}h(x^k),~~\text{with}~~h(x^k)\to h(\ov{x})~\text{and}~ \zeta^k\to \zeta\right\}.
\]
If $\func{h}{\R^n}{\R}$ is continuously differentiable, i.e., $h\in \C^1$, then
$\partial h(\ov{x})=\{\nabla h(\ov{x})\}$ \cite[page 20]{Mordukhovich2018}.
%%%%%%%%%%%%%%%%%%%%%%%%%%

A function $h\in \C^1$ is said to have $\nu$-\textit{H\"older continuous gradient} with $\nu\in (0,1]$ if there exists a constant $L>0$ such that
\begin{align*}
    \|\nabla h(x)-\nabla h(y)\| \leq L\|x-y\|^\nu, \quad \forall x,y\in\R^n,
\end{align*}
which implies the following \textit{H\"olderian descent lemma}
\begin{equation}\label{eq:HDescentLemma}
    \left\vert h(x)-h(y)-\innprod{\nabla h(y)}{x-y}\right\vert \leq \tfrac{L}{1+\nu} \|x-y\|^{1+\nu},\quad \forall x,y\in\R^n,
\end{equation}
see, e.g., \cite{nesterov2018lectures,yashtini2016global}. The class of such functions is denoted by $\C^{1,\nu}_{L}(\R^n)$. Furthermore, a function $f$ that meets this condition is called {\it weakly smooth}; see, e.g., \cite{ahookhosh2019accelerated,nesterov2015universal}.

%%%%%%%%%%%%%%%%%%%%%%%%%%
We now introduce the class of $\varrho$-\textit{weakly convex} functions \cite{nurminskii1973quasigradient}.
\begin{defin}[Weak convexity]\label{def:weakconvex}
A proper function $h: \R^n\to \Rinf$ is said to be $\varrho$-\textit{weakly convex} for $\varrho>0$ if 
$h+\tfrac{\varrho}{2}\Vert \cdot\Vert^2$ is convex.
\end{defin}

%%%%%%%%%%%%%%%%%%%%%%%%%%%%%%%%%%%%%%%%%%%%%%%%%%%%%%%%%%%%%%
The following fact, which is a special case of \cite[Lemma 1]{aragon2018accelerating}, serves as a key tool for establishing linear convergence in the subsequent sections.
\begin{fact}[\textbf{Convergence of a sequence with positive elements}]\cite[Lemma 1]{aragon2018accelerating}
\label{lem:convRate1} 
    Let $\seq{s^k}$ be a sequence in $(0, +\infty)$ converging to $0$, and let $\varsigma\ge 1$ and $\omega>0$. Suppose that for some $k_0\in \N_0$,
    \begin{equation}\label{eq:sk}
        (s^k)^{\varsigma}\leq\omega(s^k-s^{k+1}),\quad\forall k\ge k_0.
    \end{equation}
    Then the following assertions hold:
    \begin{description}
        \item[(a)] If $\varsigma=1$, then necessarily $\omega>1$, and the sequence $\seq{s^k}$ converges linearly to $0$ with rate $1-\tfrac{1}{\omega}$, i.e.,
        \begin{equation}\label{eq:sk1}
        s^{k+1} \le \big(1-\tfrac{1}{\omega}\big) s^{k}, \quad \forall k\ge k_0;
        \end{equation}
        \item[(b)] If $\varsigma>1$, then there exists a constant $\mu>0$ such that 
        \[
            0\leq s^k\leq \mu k^{-\tfrac{1}{\varsigma-1}},
        \]
        for all sufficiently large $k\in\N_{0}$.
    \end{description}
\end{fact}
%%%%%%%%%%
The following remark clarifies the connection with the original statement in \cite[Lemma 1]{aragon2018accelerating}.
\begin{rem}\label{rem:lem:convRate1}
    In \cite[Lemma 1]{aragon2018accelerating}, the condition \eqref{eq:sk} is assumed to hold for all sufficiently large $k$. Since we are using a special case of that lemma, it suffices to require \eqref{eq:sk} only for all $k\geq k_0$ for some $k_0\in \N_0$. In fact, inequality \eqref{eq:sk1} is obtained directly by setting $\varsigma=1$ in \eqref{eq:sk}. Moreover, since $s^{k+1}>0$, the inequality \eqref{eq:sk1} immediately implies that $\omega>1$.\qed
\end{rem}

%%%%%%%%%%%%%%%%%%%%%%%%%%%%%%%%%%%%%%%%%%%%%%%%%%%%%%%%%%%%%%%%%%%%%%%%%%%%%%%%%%
\subsection{{\bf Class of Kurdyka-{\L}ojasiewicz Functions}} \label{sec:KLFunc}
The establishment of global convergence and linear convergence rates for a broad class of optimization algorithms crucially relies on the Kurdyka-\L{}ojasiewicz property of the objective function (see, e.g., \cite{absil2005convergence,attouch2010proximal,bolte2007lojasiewicz,yu2022kurdyka}).
We recall the relevant definitions below.

\begin{defin}[\textbf{Kurdyka-{\L}ojasiewicz inequality}] \label{def:LojaIneqDef}
    Let $\func{h}{\R^n}{\R}$ be a lsc function.
    \begin{enumerate}
        \item\label{def:LojaIneqDef-l}
             We say that $h$ satisfies the \textit{Kurdyka-{\L}ojasiewicz} (KL) inequality at $\ov x \in \R^n$ if there exist constants $r, \xi, \tau>0$ and an exponent $\vartheta \in (0,1)$ such that
            \begin{equation}\label{eq:LojaIneqGrad}
                \big(h(x)-h(\ov x)\big)^\vartheta\leq \tau~\dist(0;\partial h(x)),
            \end{equation}
            whenever $x\in B(\ov x;r)$ and $h(\ov x)<h(x) <h(\ov x)+\xi$.
        \item\label{def:LojaIneqDef-g}
             Let $h^*:=\inf_{x\in \R^n} h(x)>-\infty$. The function $h$ is said to satisfy the \textit{global Kurdyka-{\L}ojasiewicz (global KL) inequality} if there exist constants $\tau>0$ and an exponent $\vartheta \in (0,1)$ such that
             \begin{equation}\label{eq:LojaIneqGrad:g}
                \big(h(x)-h^*\big)^\vartheta\leq \tau~\dist(0;\partial h(x)),\quad\forall x\in \R^n.
            \end{equation}
    \end{enumerate}
    The function $h$ is said to be a \textit{KL function} at $\ov{x}$ if it satisfies the KL inequality at $\ov{x}$.
Moreover, $h$ is called a \textit{global KL function} if it satisfies the global KL inequality.
\end{defin}

The inequality \eqref{eq:LojaIneqGrad} holds for a broad class of functions. {\L}ojasiewicz proved it for analytic functions \cite{lojasiewicz1963propriete,lojasiewicz1993geometrie}, and Kurdyka later extended it to differentiable functions definable in an $o$-minimal structure \cite{kurdyka1998gradients}.
The global variant \eqref{eq:LojaIneqGrad:g} with $\vartheta=\tfrac{1}{2}$ is known as the Polyak–{\L}ojasiewicz (PL) inequality \cite{Apidopoulos2022convergence,fatkhullin2022sharp}. More recently, Garrigos et al.~\cite{garrigos2023convergence} introduced the global $p$–{\L}ojasiewicz inequality with $\vartheta=1-\tfrac{1}{p}$ and showed its equivalence to the PL inequality on bounded neighborhoods of the optimal solution set.

In this paper, we use the (global) KL inequality to unify the convergence analysis of generalized descent methods. For $h\in \mathcal{C}^1$, we have 
\[
\dist(0;\partial h(x))=\Vert\nabla h(x)\Vert,
\]
which connects \eqref{eq:LojaIneqGrad} and \eqref{eq:LojaIneqGrad:g} to smooth optimization. This property underlies our results on global convergence, convergence rates, and complexity bounds (see Section~\ref{sec:genericDescMethod}).

%%%%%%%%%%%%%%%%%%%%%%%%%%%%%%%%%%%%%%%%%%%%%%%%%%%%%%%%%%%%%%%%%%%%%%%%%%%%%%%%%%
%%%%%%%%%%%%%%%%%%%%%%%%%%%%%%%%%%%%%%%%%%%%%%%%%%%%%%%%%%%%%%%%%%%%%%%%%%%%%%%%%%
\section{Generalized Descent Methods: Unified Framework} \label{sec:genericDescMethod}
In this section, we introduce a generalized descent algorithm (DEAL), presented in Algorithm~\ref{alg:deal}, for solving problem~\eqref{eq:prob} under Assumption~\ref{ass:basic}. The DEAL framework is highly flexible: it unifies well-established methods for smooth optimization, such as gradient-based algorithms, and extends naturally to nonsmooth settings through smoothing techniques, including proximal-point methods. Section~\ref{sec:InstancesDEAL} illustrates this versatility through specific instances. In particular, Examples~\ref{ex:PPM} and \ref{ex:FBM}, together with Subsection~\ref{sub:fben}, show how smoothing techniques integrate with the DEAL framework. Notably, Subsection~\ref{sub:fben} further develops high-order proximal-point methods as instances of DEAL, achieving linear convergence for all KL exponents $\vartheta \in (0,1)$.
%\vspace{4mm}
%%%%%%%%%%%%%%%%%%%%%%
\RestyleAlgo{boxruled}
\begin{algorithm}
\DontPrintSemicolon
\textbf{Initialization:} $x^0\in\R^n$,~$\rho>0$,~$\theta>1$,~ $k=0$;

    \Repeat{the stopping criteria holds
    }{
    Generate $x^{k+1}$ satisfying
    \begin{equation}\label{eq:deal}
    f(x^{k+1})\leq f(x^{k})-\rho\Vert \nabla f(x^{k})\Vert^{\theta};
    \end{equation}
    \\Set $k=k+1$;
    }
\caption{DEAL (Generalized DEscent ALgorithm) \label{alg:deal}}
\end{algorithm}

Algorithm~\ref{alg:deal} is presented in an abstract form, and several instance procedures for generating $x^{k+1}$ are discussed in Section~\ref{sec:InstancesDEAL}.
For some convergence results, we require the sequence $\seq{x^k}$ generated by Algorithm~\ref{alg:deal} to satisfy structural conditions; see Theorem~\ref{thm:GlobConv-s}~\ref{thm:GlobConv-s:3}–\ref{thm:GlobConv-s:4}. One such condition is the existence of a convergent subsequence, also known as the \textit{continuity condition} \cite{attouch2013convergence,bento2025convergence}. While not essential for all results, we formalize it below as an assumption.
%%%%%%%%%%%%%%%%%%%%%%%
\begin{ass}[Convergent subsequence]\label{ass:bo}
The sequence $\seq{x^k}$ generated by Algorithm \ref{alg:deal} has a convergent subsequence.
\end{ass}

Within the DEAL framework, we impose $\theta>1$, a condition motivated by theoretical requirements that guarantees stability and descent properties. Moreover, when the objective function of problem~\eqref{eq:prob} satisfies the KL inequality at a critical point with KL exponent $\vartheta$, our convergence analysis imposes additional restrictions on the parameter $\theta$. These restrictions reveal a fundamental link between $\vartheta$ and $\theta$, as discussed below.

We now examine the global convergence of DEAL, establishing that subsequences converge to critical points and, under additional conditions, the entire sequence converges.

%%%%%%%%%%%%%%%%%%%%%%%
\begin{thm}[Global convergence of DEAL] \label{thm:GlobConv-s}
Let the sequence $\seq{x^k}$ be generated by Algorithm \ref{alg:deal}. The following assertions hold:
    \begin{enumerate}
        \item\label{thm:GlobConv-s:1} 
            The sequence of function values $\seq{f(x^k)}$ is non-increasing;
        \item\label{thm:GlobConv-s:2}
            We have $\sum_{k=0}^\infty\Vert \nabla f(x^k)\Vert^{\theta}< \infty$ and $\displaystyle\lim_{k\to\infty} \|\nabla f(x^k)\|=0$;
        %, and $\displaystyle\lim_{k\to\infty} f(x^k)=f^*$;
        \item\label{thm:GlobConv-s:3}
            If Assumption~\ref{ass:bo} holds, every cluster point $\ov x$ of $\seq{x^k}$ is a critical point, and $\displaystyle \lim_{k\to\infty}f(x^k)= f(\ov x)$;
        \item\label{thm:GlobConv-s:4}
             If Assumption~\ref{ass:bo} holds, and there exists $c>0$ such that
            \[\|x^{k+1} - x^k\| \leq c \|\nabla f(x^k)\|^{\theta - 1},\]
            for all sufficiently large $k \in \N_{0}$, 
            and if the KL inequality holds at a cluster point $\ov{x}$ of $\seq{x^k}$, then $x^k\to \ov x$.
    \end{enumerate}
\end{thm}
%%%%%%%%
\begin{proof}
Assertion~\ref{thm:GlobConv-s:1} follows directly from \eqref{eq:deal} and $\rho>0$.
   For Assertion~\ref{thm:GlobConv-s:2}, summing both sides of \eqref{eq:deal} from $k=0$ to $k=N$ yields
   \begin{equation}\label{eq:thm:GlobConv-s}
       \rho\sum_{k=0}^N\Vert \nabla f(x^k)\Vert^{\theta}\leq f(x^0)-f(x^{N+1})\le f(x^0)- f^*,
   \end{equation}
    which implies $\sum_{k=0}^\infty\Vert \nabla f(x^k)\Vert^{\theta}< \infty$ and $\displaystyle\lim_{k\to\infty} \|\nabla f(x^k)\|=0$, verifying the desired result.
For Assertion~\ref{thm:GlobConv-s:3}, let a subsequence $\subseq{x^{k_j}}$ converge to $\ov{x}$. 
Since $f\in \mathcal{C}^1$, the gradient is continuous, and thus $\nabla f(x^{k_j})\to \nabla f(\ov x)$.
Hence, by Assertion~\ref{thm:GlobConv-s:2}, $\|\nabla f(\ov x)\|=0$ and $\ov x$ is a critical point.  
Since $\seq{f(x^k)}$ is non-increasing and bounded below, it converges, and therefore $f(x^k)\to f(\ov x)$.  

Finally, for Assertion~\ref{thm:GlobConv-s:4}, let $\ov x$ be a cluster point where the KL inequality holds. 
Let $r,\xi,\tau>0$ be the constants associated with the KL inequality at $\ov x$ (cf. Definition~\ref{def:LojaIneqDef}). 
If $\nabla f(x^k)=0$ for some $k$, the result is trivial; otherwise, assume $\nabla f(x^k)\neq 0$ for all $k$.  
Since $\ov x$ is a cluster point, Assertion~\ref{thm:GlobConv-s:3} yields that $f(x^k)\to f(\ov x)$. Moreover, there exists $\ov k$ such that $x^{\ov k}\in\Gamma$, where
\[
\Gamma := \{x\in B(\ov x;r) : f(\ov x)< f(x)< f(\ov x)+\xi\}.
\]
Additionally, as $\|x^{k+1}-x^k\|\to 0$, we have $x^k\in\Gamma$ for all $k\geq \ov k$, possibly enlarging $\ov k$.  
Define $\Delta_k:=\left(f(x^k)-f(\ov x)\right)^{1-\vartheta}$ for $k\ge \ov k$.
By concavity of $\phi(t)=t^{1-\vartheta}$ on $(0, +\infty)$, together with \eqref{eq:deal} and the KL inequality, we obtain 
 \begin{align}\label{eq:th:conKL:z1-s}
  \Delta_{k}-\Delta_{k+1} \geq 
  (1-\vartheta)\left(f(x^{k})-f(\ov x)\right)^{-\vartheta}
\left[f(x^{k})-f(x^{k +1})\right]\geq 
  \tau^{-1}\rho (1-\vartheta)
\Vert\nabla f(x^{k})\Vert^{\theta-1},
\end{align}
i.e.,
\[
\Delta_{k}-\Delta_{k+1} \geq c^{-1}\tau^{-1}\rho (1-\vartheta) \|x^{k+1}-x^{k}\|.
\]
Summing both sides of this inequality over $k$ yields
\begin{align*}
\sum_{k=0}^{\infty}\Vert x^{k+1}-x^k\Vert
&\leq \sum_{k=0}^{\ov k -1}\Vert x^{k+1}-x^k\Vert + \sum_{k=\ov k}^{\infty}\Vert x^{k+1}-x^k\Vert
\le \sum_{k=0}^{\ov k -1}\Vert x^{k+1}-x^k\Vert+\tfrac{c\tau}{\rho(1-\vartheta)}\sum_{k=\ov k}^{\infty}\left(\Delta_k-\Delta_{k+1} \right)\\
&=\sum_{k=0}^{\ov k -1}\Vert x^{k+1}-x^k\Vert+\tfrac{c\tau}{\rho(1-\vartheta)}\Delta_{\ov k}<\infty.
\end{align*}
Thus, $\{x^k\}_{k\in \N_0}$ is a Cauchy sequence and therefore converges to $\ov x$.
    \qed
\end{proof}

For the generality and motivation of the condition $\Vert x^{k+1}-x^k\Vert \leq c\Vert \nabla f(x^k)\Vert^{\theta-1}$, used in Theorem~\ref{thm:GlobConv-s}~\ref{thm:GlobConv-s:4}, we refer to Remark~\ref{rem:cononiter}.

The following theorem establishes the linear convergence of DEAL under the KL inequality at a critical point $\ov x$ with exponent $\vartheta$, focusing on the function gap, the gradient norm, and the iterate distance.
Most existing analyses yield linear rates only for $\vartheta\in(0,\tfrac{1}{2}]$; see, e.g.,
\cite{ahookhosh2021bregman,attouch2010proximal,li2018calculus,themelis2018forward,yu2022kurdyka}, Theorem~\ref{thm:LinConvWeakDeal-s} extends this to all $\vartheta \in (0,1)$.

%%%%%%%%%%%%%%%%%%%%%%%
\begin{thm}[Linear convergence rate of DEAL] 
\label{thm:LinConvWeakDeal-s}
     Let the sequence $\seq{x^k}$ be generated by Algorithm~\ref{alg:deal} under Assumption~\ref{ass:bo}. Let us assume that there exists $c>0$ such that 
     \[\Vert x^{k+1}-x^k\Vert \leq c\Vert \nabla f(x^k)\Vert^{\theta-1},\]
     for all sufficiently large $k\in\N$. If $\ov x$ is a cluster point at which the KL inequality holds with exponent $\vartheta$ and constant $\tau>0$, then the following assertions hold:
    \begin{enumerate}
        \item\label{thm:LinConvWeakDealFun-s:0}
            If $\theta<\tfrac{1}{\vartheta}$, the sequences $\seq{f(x^k)-f(\ov x)}$, $\seq{\|\nabla f(x^k)\|}$, and $\seq{\|x^k-\ov x\|}$ become zero in finitely many iterations;
        \item\label{thm:LinConvWeakDealFun-s:1}
            If $\theta=\tfrac{1}{\vartheta}$, the sequences $\seq{f(x^k)-f(\ov x)}$, $\seq{\|\nabla f(x^k)\|}$, and $\seq{\|x^k-\ov x\|}$ converge linearly to $0$, i.e.,
        \begin{equation}\label{eq:convRateFun-s}
            f(x^{k+1})-f(\ov x)\leq q (f(x^k)-f(\ov x)),
        \end{equation}
        \begin{equation}\label{eq:LinIneqNablaDeal-s}
            \|\nabla f(x^k)\|\le \Big( \tfrac{f(x^0)-f(\ov x)}{\rho}\Big)^{\tfrac{1}{\theta}} q^{\tfrac{k}{\theta}},
        \end{equation}
        \begin{equation}\label{eq:LinIneqIteDealp-s}
        \|x^k-\ov x\| \le \tfrac{c}{1-q^{\tfrac{\theta-1}{\theta}}}\Big( \tfrac{f(x^0)-f(\ov x)}{\rho}\Big)^{\tfrac{\theta-1}{\theta}} q^{\tfrac{k(\theta-1)}{\theta}},
        \end{equation}
        as $k\to \infty$, where $q:= 1- \tfrac{\rho}{\tau^{\theta}}\in (0, 1)$;
        \item\label{thm:LinConvWeakDealFun-s:2}
            If $\theta>\tfrac{1}{\vartheta}$, there exists $\mu>0$ such that 
            \begin{equation}\label{eq:convRateFun2-s}
                0\leq f(x^k)-f(\ov x)\leq \mu k^{-\tfrac{1}{\vartheta \theta-1}},
            \end{equation}
            \begin{equation}\label{eq:SubLinIneqNablaDeal-s}
            \|\nabla f(x^k)\|\leq \big(\tfrac{\mu}{\rho}\big)^{\tfrac{1}{\theta}} k^{-\tfrac{1}{\theta(\vartheta \theta-1)}},
            \end{equation}
            \begin{equation}\label{eq:SubLinIneqItarDeal-s}
            \|x^k-\ov x\|\leq c  \tau\rho^{-1}(1-\vartheta)^{-1}\mu^{1-\vartheta} k^{-\tfrac{1-\vartheta}{\vartheta \theta-1}},
            \end{equation}
           for all sufficiently large $k\in\N_{0}$.
    \end{enumerate}
\end{thm}

%%%%%%%%%%
\begin{proof}
By Theorem~\ref{thm:GlobConv-s}~\ref{thm:GlobConv-s:4}, $x^k\to \ov x$ and $f(x^k)\to f(\ov x)$.
From the KL inequality, there exists $k_0\in \N$ such that
\[
\big(f(x^{k}) -f(\ov x)\big)^{\vartheta}\le \tau \|\nabla f(x^k)\|, \quad\quad \forall k\ge k_0.
\]
Combining this with \eqref{eq:deal} gives
    \begin{equation}\label{eq:decIneqCSS1}
        \tfrac{\rho}{\tau^{\theta}} \big(f(x^{k}) -f(\ov x)\big)^{\vartheta\theta}\le \rho \|\nabla f(x^k)\|^{\theta} \le f(x^{k}) - f(x^{k+1}) = \big(f(x^{k}) -f(\ov x)\big) - \big(f(x^{k+1})-f(\ov x)\big),
    \end{equation}
    for all $k\ge k_0$.\\
    \ref{thm:LinConvWeakDealFun-s:0} By indirect proof, assume that $x^{k+1}\neq x^k$ for all $k\in \N_0$. It follows from \eqref{eq:decIneqCSS1} that
    \[
     \big(f(x^k)-f(\ov x)\big)^{1-\vartheta \theta}\geq \tfrac{\rho}{\tau^\theta}>0,
     \quad\quad \forall k\ge k_0,
    \]
    making a contradiction with $f(x^k)\to f(\ov x)$. Thus, all three sequences reach zero in a finite number of iterations.\\
    \ref{thm:LinConvWeakDealFun-s:1}-\ref{thm:LinConvWeakDealFun-s:2} Considering \eqref{eq:decIneqCSS1} and setting $s^k:=f(x^k)-f(\ov x)$, $\varsigma=\vartheta\theta$, and $\omega:= \tfrac{\tau^{\theta}}{\rho}$, we obtain a recurrence of the form \eqref{eq:sk}. 
    Applying Fact~\ref{lem:convRate1} and Remark~\ref{rem:lem:convRate1} yields $\omega >1$, $q\in (0,1)$, and the rates in \eqref{eq:convRateFun-s} and \eqref{eq:convRateFun2-s}.  Moreover, from \eqref{eq:deal}, 
    \begin{equation*}
        \rho \|\nabla f(x^k)\|^\theta \le f(x^{k})-f(x^{k+1}) \le f(x^k) -f(\ov x),
    \end{equation*}
   which, together with \eqref{eq:convRateFun-s} and \eqref{eq:convRateFun2-s}, implies \eqref{eq:LinIneqNablaDeal-s} and \eqref{eq:SubLinIneqNablaDeal-s}. To prove \eqref{eq:LinIneqIteDealp-s}, using 
    $\Vert x^{k+1}-x^k\Vert \leq c\Vert \nabla f(x^k)\Vert^{\theta-1}$ and \eqref{eq:LinIneqNablaDeal-s}, we obtain
    \begin{align*}
        \|x^k - \ov x\| &\le \sum_{i\ge k} \|x^i -x^{i+1}\|\le \sum_{i\ge k} c\|\nabla f(x^i)\|^{\theta-1}
        \le c\Big( \tfrac{f(x^0)-f(\ov x)}{\rho}\Big)^{\tfrac{\theta-1}{\theta}} \sum_{i\ge k}   q^{\tfrac{i(\theta -1)}{\theta}}
        = c\Big( \tfrac{f(x^0)-f(\ov x)}{\rho}\Big)^{\tfrac{\theta-1}{\theta}}  \tfrac{q^{\tfrac{k(\theta-1)}{\theta}}}{1-q^{\tfrac{\theta -1}{\theta}}},
    \end{align*}
    for all $k\ge k_0$.
Let us set $\theta>\tfrac{1}{\vartheta}$. 
Arguing as in the proof of Theorem~\ref{thm:GlobConv-s}~\ref{thm:GlobConv-s:4} (see in particular the derivation of
\eqref{eq:th:conKL:z1-s}), we obtain
\begin{equation}\label{eq:thm:LinConvWeakDealFun-s:1:a}
    \Delta_{k}-\Delta_{k+1} \geq \tau^{-1}\rho (1-\vartheta)\Vert\nabla f(x^{k})\Vert^{\theta-1},
\end{equation}
where $\Delta_k:=\left(f(x^k)-f(\ov x)\right)^{1-\vartheta}$. 
Moreover, by \eqref{eq:convRateFun2-s}, there exists a constant $\mu>0$ such that, for all
sufficiently large $k\in\N_{0}$,
\[
\Delta_k\leq \mu^{1-\vartheta} k^{-\tfrac{1-\vartheta}{\vartheta \theta-1}}.
\]
Together with \eqref{eq:thm:LinConvWeakDealFun-s:1:a}, this estimate yields, for all sufficiently
large $k\in\N_{0}$,
    \begin{align*}
        \|x^k - \ov x\| \le \sum_{i\ge k} c\|\nabla f(x^i)\|^{\theta-1} &\le 
        c  \tau\rho^{-1}(1-\vartheta)^{-1}\sum_{i\ge k} (\Delta_{i}-\Delta_{i+1})
        \leq c  \tau\rho^{-1}(1-\vartheta)^{-1}\Delta_{k}
        \\&
        \le c  \tau\rho^{-1}(1-\vartheta)^{-1}\mu^{1-\vartheta} k^{-\tfrac{1-\vartheta}{\vartheta \theta-1}},
    \end{align*}
which completes the proof.
     \qed
\end{proof}
%%%%%%%%%%%

The following subsection analyzes the convergence of DEAL under the global KL property and derives explicit iteration complexity bounds.

\subsection{{\bf Convergence and Complexity Under Global KL Conditions}}

In Theorems~\ref{thm:GlobConv-s} and \ref{thm:LinConvWeakDeal-s}, we have established the global and linear convergence of DEAL at a critical point under the KL property. Imposing the global KL condition yields analogous results, with a stronger guarantee of convergence to a global optimal solution. Furthermore, the global KL framework enables the complexity analysis of DEAL.

%%%%%%%%%%%%%%%%%%%%%%%
\begin{thm}[Global convergence of DEAL under global KL] \label{thm:GlobConv}
    Suppose $f$ satisfies the global KL inequality with exponent $\vartheta$ and constant $\tau>0$, and let $\seq{x^k}$ be generated by Algorithm~\ref{alg:deal}. The following assertions hold:
    \begin{enumerate}
        \item\label{thm:GlobConv-2}
            $\displaystyle\lim_{k\to\infty} f(x^k)=f^*$;
        \item\label{thm:GlobConv-3}
            If Assumption~\ref{ass:bo} holds, every cluster point of $\seq{x^k}$ belongs to $\mathcal{X}^*$;
        \item\label{thm:GlobConv-4}
            If there exists $c>0$ such that         
            $\Vert x^{k+1}-x^k\Vert \leq c\Vert \nabla f(x^k)\Vert^{\theta-1}$ for all $k\in \N_0$, then $x^k\to x^*$ for some $x^*\in \mathcal{X}^*$.
    \end{enumerate}
\end{thm}
%%%%%%%%
\begin{proof}
From the global KL inequality,
\begin{equation}\label{eq:th:globconv:klI:a}
    \tau\Vert \nabla f(x^k)\Vert\geq \left(f(x^k) - f^*\right)^{\vartheta},\quad\quad \forall k\in \N,
\end{equation}
and since $\displaystyle\lim_{k\to\infty} \|\nabla f(x^k)\|=0$ (see Theorem~\ref{thm:GlobConv-s}~\ref{thm:GlobConv-s:2}), it holds that $f(x^k)\to f^*$, verifying Assertion~\ref{thm:GlobConv-2}. 
Let $\ov x$ be a cluster point of the sequence $\seq{x^k}$. 
By Theorem~\ref{thm:GlobConv-s}~\ref{thm:GlobConv-s:3}, we have $f(x^k)\to f(\ov{x})$.
Moreover, by Assertion~\ref{thm:GlobConv-2}, it holds that  $f(x^k)\to f^*$. 
Consequently, $f(\ov x)=f^*$, which implies that $\ov x\in\mathcal X^*$, confirming
Assertion~\ref{thm:GlobConv-3}.
For Assertion~\ref{thm:GlobConv-4}, the argument parallels the proof of Theorem \ref{thm:GlobConv-s}~\ref{thm:GlobConv-s:4}. From \eqref{eq:th:globconv:klI:a},
\[
\Vert x^{k+1}-x^k\Vert \le \tfrac{c\tau}{\rho(1-\vartheta)}(\Delta_k-\Delta_{k+1}),\quad\quad \forall k\in \N,
\]
for $\Delta_k := (f(x^k)-f^*)^{1-\vartheta}$, i.e.,
\begin{align*}
\sum_{k=0}^{\infty}\Vert x^{k+1}-x^k\Vert\leq   \tfrac{c\tau}{\rho(1-\vartheta)}\sum_{k=0}^{\infty}\left(\Delta_k-\Delta_{k+1} \right)
=\tfrac{c\tau}{\rho(1-\vartheta)}\Delta_{0}<\infty.
\end{align*}
Hence $\seq{x^k}$ is a Cauchy sequence and, therefore, converges to some $x^*\in\mathcal{X}^*$. 
    \qed
\end{proof}
%%%%%%%%%%%
This result paves the way for establishing the linear convergence of DEAL.
%%%%%%%%%%%%%%%%%%%%%%%%%%%%%%%%%%%
\begin{cor}[Linear convergence of DEAL under global KL property]\label{cor:onconvinglob}
     Let $f$ satisfy the global KL inequality with exponent $\vartheta$ and constant $\tau>0$, and let the sequence $\seq{x^k}$ be generated by Algorithm~\ref{alg:deal}. Suppose there exists $c>0$ such that          
            \[\Vert x^{k+1}-x^k\Vert \leq c\Vert \nabla f(x^k)\Vert^{\theta-1},\]
            for all $k\in \N_0$.
            Then, there exists $x^*\in \X^*$ such that the following assertions hold:
            \begin{enumerate}
                \item\label{cor:onconvinglob-0}
            If $\theta<\tfrac{1}{\vartheta}$, the sequences $\seq{f(x^k)-f(x^*)}$, $\seq{\|\nabla f(x^k)\|}$, and $\seq{\|x^k-x^*\|}$ become zero in finitely many iterations;
                \item\label{cor:onconvinglob-1}
                    If $\theta=\tfrac{1}{\vartheta}$, the sequences $\seq{f(x^k)-f(x^*)}$, $\seq{\|\nabla f(x^k)\|}$, and $\seq{\|x^k-x^*\|}$ converge linearly to $0$, i.e.,
                    \eqref{eq:convRateFun-s}-\eqref{eq:LinIneqIteDealp-s} hold
                    for all $k\in \N_0$ with $\ov x$ replaced by $x^*$;

                \item\label{cor:onconvinglob-2}
                    If $\theta>\tfrac{1}{\vartheta}$, there exists $\mu>0$ such that 
                    \eqref{eq:convRateFun2-s}-\eqref{eq:SubLinIneqItarDeal-s} hold
                   for all sufficiently large $k\in\N_{0}$, with $\ov x$ replaced by $x^*$.

            \end{enumerate}   
\end{cor}
\begin{proof}
    This follows directly from Theorems~\ref{thm:LinConvWeakDeal-s} and \ref{thm:GlobConv}. \qed
\end{proof}

We now analyze the complexity of DEAL. Fix constants $\varepsilon,\rho,\tau,\theta>0$ such that $q:=1-\tfrac{\rho}{\tau^\theta}>0$, and define the following operator for $x,y>0$:
\begin{align*}
(x,y) \mapsto\K(x, y):=\left\lceil\tfrac{y\log(\varepsilon^{-1}) + \log\left(x\right)}{\log\big(q^{-1}\big)}+1\right\rceil.
\end{align*}

%%%%%%%%%%%%%%%%%%%%%%%
\begin{thm}[Complexity analysis of DEAL under the global KL property] 
\label{thm:LinConvWeakDealFun}
    Let $f$ satisfy the global KL inequality with exponent $\vartheta$ and constant $\tau>0$, and let $\seq{x^k}$ be generated by Algorithm~\ref{alg:deal} with $\theta=\tfrac{1}{\vartheta}$. Suppose there exists $c>0$ such that 
    \[\Vert x^{k+1}-x^k\Vert \leq c\Vert \nabla f(x^k)\Vert^{\theta-1},\]
    for all $k\in\N_0$.  
Then, for any accuracy $\varepsilon>0$, there exists $x^*\in\X^*$ such that the following assertions hold:
    \begin{enumerate}
    \item\label{thm:LinConvWeakDealFun-1}
        The number of iterations to reach
        $f(x^k)-f^*\leq \varepsilon$, denoted $\Ni^f(\varepsilon)$, satisfies
        \[\Ni^f(\varepsilon)\leq \K(f(x^0)-f^*, 1);\]

         \item\label{thm:LinConvWeakDealFun-2}
        The number of iterations to reach $\|\nabla f(x^k)\|\leq \varepsilon$, denoted $\Ni(\varepsilon)$, satisfies
        \[\Ni(\varepsilon)\le \K\big(\rho^{-1}(f(x^0)-f^*), \theta\big);\]

        \item\label{thm:LinConvWeakDealFun-3}
        The number of iterations to reach $\|x^k-x^*\|\leq \varepsilon$, denoted $\Ni^x(\varepsilon)$, satisfies 
        \[\Ni^x(\varepsilon)\le \K\left(s\rho^{-1}(f(x^0)-f^*), \tfrac{\theta}{\theta-1}\right),\]
        with $s:=\Big(\tfrac{c}{1-q^{\tfrac{\theta-1}{\theta}}}\Big)^{\tfrac{\theta}{\theta-1}}$.
        \end{enumerate}
\end{thm}
%%%%%%%%%%
\begin{proof}
\ref{thm:LinConvWeakDealFun-1} 
From Corollary~\ref{cor:onconvinglob}~\ref{cor:onconvinglob-1}, the inequality
\eqref{eq:convRateFun-s} holds for all $k\in \N_0$ with $\ov{x}$ replaced by $x^*$. Setting $\widehat{\K}_1:=\K(f(x^0)-f^*, 1)$ gives
    \begin{align*}
       f(x^{\widehat{\K}_1})-f^*\leq q \left(f(x^{\widehat{\K}_1-1})-f^*\right) \leq\ldots\leq q^{\widehat{\K}_1} \left(f(x^{0})-f^*\right) \le  \varepsilon,
    \end{align*}
    confirming that $\Ni^f(\varepsilon)\le \widehat{\K}_1$.
\\
\ref{thm:LinConvWeakDealFun-2} 
It follows by combining inequality \eqref{eq:LinIneqNablaDeal-s}, which holds for all $k\in \N_0$ with $\ov{x}$ replaced by $x^*$, with  
the definition of $\K\big(\rho^{-1}(f(x^0)-f^*), \theta\big)$.
%\eqref{eq:UppercomNiDeal}.
    \\
 \ref{thm:LinConvWeakDealFun-3} It is concluded from \eqref{eq:LinIneqIteDealp-s}, which holds for all $k\in \N_0$ with $\ov{x}$ replaced by $x^*$, together with the definition of 
 $\K\left(s\rho^{-1}(f(x^0)-f^*), \tfrac{\theta}{\theta-1}\right)$, giving our desired result.
    \qed
\end{proof}
%%%%%%%%%%%

It is worth noting that certain convergence guarantees for the gradient norm can be obtained directly from the descent property \eqref{eq:deal}, without invoking the KL framework. However, imposing the global KL property significantly strengthens these guarantees by yielding substantially sharper complexity bounds. The following remark formalizes this observation.
%%%%%%%%%%%%%%%%%%%%%%%
\begin{rem}\label{rem:LinConvGrad}
As shown in the proof of Theorem~\ref{thm:GlobConv-s}~\ref{thm:GlobConv-s:2}, inequality \eqref{eq:thm:GlobConv-s} follows immediately from \eqref{eq:deal}, implying
$\sum_{k=0}^\infty\Vert \nabla f(x^k)\Vert^{\theta}< \infty$ and 
$\displaystyle\lim_{k\to\infty} \|\nabla f(x^k)\|=0$, without requiring the KL inequality. Moreover, \eqref{eq:thm:GlobConv-s} yields
 \[
    \rho N \min_{0\leq k\leq N-1}\|\nabla f(x^k)\|^{\theta}\leq f(x^0)-f^*,
    \]
and hence
    \begin{equation*}\label{eq:convrateg}
             \min_{0\leq k\leq N-1}\|\nabla f(x^k)\|\leq \left(\tfrac{f(x^0)-f^*}{\rho N}\right)^{\nicefrac{1}{\theta}},\quad \forall N\in\N.
    \end{equation*}
    Consequently, for a given accuracy parameter $\varepsilon>0$, we have
   \begin{align*}
           \min_{0\leq k\leq \K-1}\|\nabla f(x^k)\|\leq \left(\tfrac{f(x^0)-f^*}{\rho \K}\right)^{\nicefrac{1}{\theta}}\leq \varepsilon,
    \end{align*}
with $\K:=\left\lceil\rho^{-1} \varepsilon^{-\theta}\big(f(x^0)-f^*\big)+1\right\rceil$.
   In other words, the condition
     $\|\nabla f(x^k)\|\leq \varepsilon$ is satisfied within 
  $O\big(\varepsilon^{-\theta}\big)$
iterations in this case. By contrast, Theorem~\ref{thm:LinConvWeakDealFun} shows that, under the global KL inequality, the iteration complexity improves to
$O\big(\log(\varepsilon^{-1})\big)$.\qed
\end{rem}

We conclude this section with a remark on the global KL condition.  

\begin{rem}
If the global KL inequality \eqref{eq:LojaIneqGrad:g} holds on the sublevel set
$\mathcal{L}(\lambda):=\{x \in \R^n \mid f(x)\leq \lambda\}$ for some $\lambda>0$,
and the initial point satisfies $x^0 \in \mathcal{L}(\lambda)$. Then, all convergence and complexity results
established in this section remain valid, with convergence guaranteed to a global minimizer.\qed
\end{rem}

\section{Generalized Descent Directions and Instances of DEAL}\label{sec:InstancesDEAL}
We have already analyzed properties of the sequence $\seq{x^k}$ generated by DEAL without specifying a constructive mechanism for generating the iterates, in particular the role of search directions. To improve both the generality and efficiency of DEAL, we now incorporate descent directions, which will play a crucial role in the following subsections.

Next, we extend the notion of (sufficient) descent direction.  

\begin{defin}[Generalized descent direction]\label{def:G-descent-dir}
A direction $d^k\in \R^n$ is a \textit{generalized (sufficient) descent direction} if there exist $\beta>-1$ and a (sufficient) descent direction $\ov d^k$ such that
    \begin{equation}\label{eq:genDescDir}
        d^k= \|\nabla f(x^k)\|^\beta \ov d^k.
    \end{equation}
\end{defin}

From Definition~\ref{def:G-descent-dir} and \eqref{eq:sufDescCond}, if $d^k$ is a generalized sufficient descent direction, then it holds that
\begin{equation}\label{eq:genSufDescCond}
    \innprod{\nabla f(x^k)}{d^k}\leq -c_1\|\nabla f(x^k)\|^{2+\beta}, \quad \|d^k\|\leq c_2 \|\nabla f(x^k)\|^{1+\beta}.
\end{equation}
The parameter $\beta$ in Definition~\ref{def:G-descent-dir} adds flexibility by scaling the descent strength with the gradient norm, potentially improving performance; see Section~\ref{sec:preNumExper}. In the forthcoming DEAL-C and DEAL-A algorithms, $\beta$ is linked to the $\nu$-H\"{o}lder condition and the KL exponent $\vartheta$, enabling linear convergence. Furthermore, as discussed in the following remark, setting $\theta=2+\beta$ in Algorithm~\ref{alg:deal} ensures
   \begin{align}\label{eq:MainEquality}
        \Vert x^{k+1}-x^k\Vert \leq c\Vert \nabla f(x^k)\Vert^{\theta-1},
    \end{align}
for some $c>0$, so this condition holds automatically and need not be imposed separately.

\begin{rem}\label{rem:cononiter}
    In Theorem~\ref{thm:GlobConv-s}~\ref{thm:GlobConv-s:4}, we assume condition \eqref{eq:MainEquality}.  
This requirement is automatically satisfied in standard settings. Indeed, consider the update
    $x^{k+1} = x^k + \alpha_k d^k$, where $d^k=\|\nabla f(x^k)\|^{\beta} \ov d^k$ with $\beta>-1$ and $\ov d^k$ is a sufficient descent direction satisfying \eqref{eq:sufDescCond}. If $\alpha_{k}\le \alpha_{\max}$ for some $\alpha_{\max}>0$ and 
    $\theta=2+\beta$, then
    \[
    \|x^{k+1}-x^{k}\| = \alpha_k \|d^k\| \le \alpha_{\max} c_2 \|\nabla f(x^k)\|^{\theta-1}.
    \]
This ensures that \eqref{eq:MainEquality} holds with $c=\alpha_{\max} c_2$.\qed 
\end{rem}

%%%%%%%%%%%%%%%%%%%%%%%%%%%%%%%%%%%%%%%%%%%%%%%%%%%%%%%%%%%%%%%%%%%%%%%%%%%%%%%%%%
\subsection{{\bf DEAL with Constant Step-Size}} \label{sec:descMethodCSS}
In this subsection, we present DEAL with a constant step-size (DEAL-C) as a basic instance of the DEAL framework. The structure and well-definedness of this algorithm rely on the $\nu$-H\"older continuity of the gradient of the objective function $f$, as stipulated in the following assumption.
\begin{ass}\label{ass:Holder}
        The objective function $f$ in \eqref{eq:prob} has a $\nu$-H\"older continuous gradient with $\nu\in(0,1]$ and constant $L>0$, i.e., $f\in \C^{1,\nu}_L(\R^n)$.
\end{ass}

Let $d^k\in\R^n$ satisfy the generalized sufficient descent condition \eqref{eq:genSufDescCond} with $\beta=\tfrac{1-\nu}{\nu}$ and constants $c_1,c_2>0$. Consider the update rule $x^{k+1}=x^k+\alpha_k d^k$ with
\begin{equation}\label{eq:cons:al}
\alpha_k\in \left( 0, \big(\tfrac{c_1(1+\nu)}{c_2^{1+\nu}L}\big)^{\nicefrac{1}{\nu}}\right].
\end{equation}
From Assumption~\ref{ass:Holder}, the H\"olderian descent lemma, condition \eqref{eq:genSufDescCond}, and $\beta=\tfrac{1-\nu}{\nu}$, it follows that
    \begin{equation}\label{eq:decIneq}
        \begin{split}
            f(x^{k+1})=f(x^k+\alpha_k d^k) &\leq f(x^k)+ \alpha_k \innprod{\nabla f(x^k)}{d^k}+\tfrac{L}{1+\nu} \alpha_{k}^{1+\nu} \|d^k\|^{1+\nu}\\
            &\leq f(x^k)- c_1\alpha_{k} \|\nabla f(x^k)\|^{2+\beta}+\tfrac{c_2^{1+\nu} L}{1+\nu} \alpha_{k}^{1+\nu} \|\nabla f(x^k)\|^{(1+\beta)(1+\nu)}\\
            &\leq f(x^k)- \alpha_{k} \left(c_1 -\tfrac{c_2^{1+\nu} L}{1+\nu} \alpha_{k}^\nu\right) \|\nabla f(x^k)\|^{\tfrac{1+\nu}{\nu}}.
        \end{split}
    \end{equation}
    By \eqref{eq:cons:al}, we obtain $\alpha_{k} \big(c_1 -\tfrac{c_2^{1+\nu} L}{1+\nu} \alpha_{k}^\nu\big)\geq 0$, ensuring $f(x^{k+1})\leq f(x^k)$. Moreover, the maximum decrease is achieved by maximizing the concave function $\alpha \mapsto  c_1\alpha -\tfrac{c_2^{1+\nu} L}{1+\nu} \alpha^{\nu+1}$ which attains its peak at 
    $\alpha=\big(\tfrac{c_1}{c_2^{1+\nu}L}\big)^{\nicefrac{1}{\nu}}$, yielding
   \begin{equation}\label{eq:maxst}
        f(x^{k+1}) \leq f(x^k)- \tfrac{c_1 \alpha\nu}{1+\nu} \|\nabla f(x^k)\|^{1+\tfrac{1}{\nu}}.
    \end{equation}
    
We now formalize DEAL with constant step-size in Algorithm~\ref{alg:dealConstantSS}. In Lemma \ref{lem:dealc:deal} we show that DEAL-C fits within the general DEAL framework through suitable choices of $\theta$ and $\rho$.
%%%%%%%%%%%%%%%%%%%%%%
\RestyleAlgo{boxruled}
\begin{algorithm}[h]
\DontPrintSemicolon
\textbf{Initialization:} $x^0\in\R^ n$,~ $k=0$;

    \Repeat{the stopping criteria holds
    }{
    Choose a direction $d^k$ as \eqref{eq:genDescDir} satisfying \eqref{eq:genSufDescCond} with $\beta=\tfrac{1-\nu}{\nu}$;\;
    Set the constant step-size as $\alpha=\Big(\tfrac{c_1}{c_2^{1+\nu}L}\Big)^{\nicefrac{1}{\nu}}$;\;
    Set $x^{k+1}=x^k+\alpha d^k$ and $k=k+1$;
    }
\caption{DEAL-C (DEAL with constant step-size) \label{alg:dealConstantSS}}
\end{algorithm}
%%%%%%%%%%%%%%%

%%%%%%%%%%%%%%%%%%%%%%%%%%%%%%%%%%%%%%%%%%
\begin{lem}[DEAL-C as an instance of DEAL]\label{lem:dealc:deal}
Let Assumption~\ref{ass:Holder} hold and $\seq{x^k}$ be generated by Algorithm \ref{alg:dealConstantSS}.
Then, this sequence can be reproduced by Algorithm~\ref{alg:deal}, by setting $\theta =\beta+2=1+\tfrac{1}{\nu}$
and  $\rho = \tfrac{c_1 \alpha\nu}{1+\nu}$, where $\alpha= \big(\tfrac{c_1}{c_2^{1+\nu}L}\big)^{\nicefrac{1}{\nu}}$.
\end{lem}
%%%%%%%%%

%%%%%%%%%%%%%
\begin{proof}
From the construction of Algorithm \ref{alg:dealConstantSS}, we have $\theta>1$, $\rho>0$, and 
from \eqref{eq:maxst}, $x^{k+1}$ is generated such that satisfies \eqref{eq:deal}.
\qed
\end{proof}
%%%%%%%%%%%

As a direct consequence, DEAL-C inherits the global convergence, linear rates, and complexity guarantees established for the general DEAL framework.

\begin{cor}[Convergence and complexity analysis of DEAL-C]\label{cor:conDEALc}
Let Assumption \ref{ass:Holder} hold and $\seq{x^k}$ be generated by Algorithm \ref{alg:dealConstantSS}.
Then, the results of Theorems~\ref{thm:GlobConv-s},~\ref{thm:LinConvWeakDeal-s},~\ref{thm:GlobConv},~and~\ref{thm:LinConvWeakDealFun} and Corollary~\ref{cor:onconvinglob} hold
by setting  $\theta =\beta+2=1+\tfrac{1}{\nu}$, $c = c_2 \alpha$,
and  $\rho = \tfrac{c_1 \alpha\nu}{1+\nu}$, where $\alpha= \big(\tfrac{c_1}{c_2^{1+\nu}L}\big)^{\nicefrac{1}{\nu}}$.
\end{cor}
\begin{proof}
The result follows directly from Lemma~\ref{lem:dealc:deal} and Remark~\ref{rem:cononiter} with $\alpha_{\max}=\alpha$. \qed
\end{proof}

A direct consequence of Corollary~\ref{cor:conDEALc} together with Theorem~\ref{thm:LinConvWeakDeal-s}~\ref{thm:LinConvWeakDealFun-s:1} is that, in order to achieve linear convergence in this setting, it is necessary that $\vartheta = \tfrac{\nu}{1+\nu}$. As a convex example satisfying this relation, let us consider the linear inverse problem given by
\begin{equation}\label{eq:PRP}
    \min_{x\in\mathbb{R}^n} \; \tfrac{1}{\delta}\|Ax-b\|^\delta,
\end{equation}
where $A\in \R^{m\times n}$ is a full rank matrix with smallest singular value $\sigma_{\min}>0$, $b\in \R^m$, and $\delta\in (1,2]$.
In the following proposition, we establish that the objective function in (\ref{eq:PRP}) satisfies the Kurdyka-{\L}ojasiewicz (KL) inequality with exponent $\vartheta=1-\frac{1}{\delta}$ and possesses a  $(\delta-1)$-H\"older gradient. For a nonconvex example, see Section \ref{sec:preNumExper}.

\begin{prop}\label{prop:invprob}
   Let  $f$ be the function defined by \eqref{eq:PRP}. Then, $f$ holds the global KL inequality with exponent $\vartheta=1-\frac{1}{\delta}$ and constant $\tau=\tfrac{1}{\sigma_{\min}\delta^{1-\tfrac{1}{\delta}}}$.  Moreover, the function $f$ has $(\delta-1)$-H\"older gradient with constant $2^{2-\delta}\|A\|^\delta$.
\end{prop}
\begin{proof}
    Without loss of generality, let $x\in \R^n$ such that $Ax-b\neq 0$. It holds that
    \[
    \tau \|\nabla f(x)\| = \tau \Big\|\|Ax-b\|^{\delta-2} A^T(Ax-b)\Big\| \geq \tau \sigma_{\min} \|Ax-b\|^{\delta-1} = \tfrac{1}{\delta^{\vartheta}}\|Ax-b\|^{\delta\vartheta} \ge (f(x)-f^*)^\vartheta,
    \]
    verifying the global KL property. 
    To prove $(\delta-1)$-H\"older gradient property, it follows from \cite[Theorem 6.3]{rodomanov2020smoothness} that
    \begin{align}\label{eq:HolderNorm}
        \Big\|\|z\|^{\delta-2}z - \|w\|^{\delta-2}w\Big\|\le 2^{2-\delta} \|z-w\|^{\delta-1},
    \end{align}
    for all $z,w\in \R^m$. Thus, for any $x, y \in \R^n$, we obtain
    \begin{align*}
        \|\nabla f(x) - \nabla f(y)\| &=
        \Big \| \|Ax-b\|^{\delta-2} A^T(Ax-b) - \|Ay-b\|^{\delta-2} A^T(Ay-b) \Big\| \\
        &\le \|A\| \Big \| \|Ax-b\|^{\delta-2} (Ax-b) - \|Ay-b\|^{\delta-2} (Ay-b) \Big\|\\
        & \le 2^{2-\delta}\|A\|  \big\|(Ax-b)-(Ay-b)\big\|^{\delta-1}\\
        &\le 2^{2-\delta}\|A\|^\delta \|x-y\|^{\delta-1},
    \end{align*}
    ensuring the $(\delta-1)$-H\"older gradient property of $f$ with constant $2^{2-\delta}\|A\|^\delta$.\qed
\end{proof}
%%%%%%%%%%%

%%%%%%%%%%%%%%%%%%%%%%%%%%%%%%%%%%%%%%%%%%%%%%%%%%%%%%%%%%%%%%%%%%%%%%%%%%%%%%%%%%
\subsection{{\bf DEAL with Armijo Line Search}} \label{sec:descMethodALS}
As a second instance of the DEAL framework, we now formalize DEAL with Armijo line search (DEAL-A), presented in Algorithm~\ref{alg:dealArmijo}.

\vspace{4mm}
%%%%%%%%%%%%%%%%%%%%%%
\RestyleAlgo{boxruled}
\begin{algorithm}[H]
\DontPrintSemicolon
\textbf{Initialization:} $x^0\in\R^ n$,~ $\ov\alpha>0$,~ $\sigma\in (0,1)$,~ $\eta\in (0,1)$,~ $k=0$;

    \Repeat{
        the stopping criterion holds
    }{
    Choose a direction $d^k$ as \eqref{eq:genDescDir} satisfying \eqref{eq:genSufDescCond} with $\beta=\tfrac{1-\nu}{\nu}$;\;
    Set $\ov \alpha_k=\ov\alpha$, $p=0$, and  $\ov x^{k+1}=x^k+\ov\alpha_k d^k$;\;
    \While{$f(\ov x^{k+1})>f(x^k)+\sigma \ov\alpha_k \innprod{\nabla f(x^k)}{d^k}$}{
        Set $p=p+1$, $\ov \alpha_k=\eta^p \ov\alpha$, and $\ov x^{k+1}=x^k+\ov\alpha_k d^k$;
    }
    Set $p_k=p$, $\alpha_k=\ov \alpha_k$, $x^{k+1}=\ov x^{k+1}$, and $k=k+1$;
    }
\caption{DEAL-A (DEAL with Armijo line search) \label{alg:dealArmijo}}
\end{algorithm}
%%%%%%%%%%%%%%%
\vspace{4mm}

%%%%%%%%%%%%
The termination of the loop in Steps~5–7 of Algorithm~\ref{alg:dealArmijo} after finitely many backtracking steps is a standard result \cite{nocedal2006numerical}. 
For each $k\in \N_0$, Algorithm~\ref{alg:dealArmijo} guarantees
\[
f(x^{k+1})\leq f(x^k)+\sigma \alpha_k \innprod{\nabla f(x^k)}{d^k}.
\]
If $d^k$ satisfies the generalized sufficient descent conditions \eqref{eq:genSufDescCond}, then
\begin{align}\label{eq:DealArmio1}
    f(x^{k+1})\leq f(x^k)-\sigma \alpha_k c_1\|\nabla f(x^k)\|^{2+\beta}.
\end{align}
Consequently, if the step-size $\alpha_k$ is bounded below by some $\tilde{\alpha}>0$, we obtain
\begin{align}\label{eq:DealArmio2}
    f(x^{k+1})\leq f(x^k)-\rho\|\nabla f(x^k)\|^{2+\beta},
\end{align}
where $\rho=\sigma \tilde{\alpha} c_1$. In this case, Algorithm~\ref{alg:dealArmijo} fits into the DEAL framework. 
The next lemma establishes a sufficient condition ensuring such a positive lower bound on the step-size.

%%%%%%%%%%%%%%%%%%%%%%%%%%%%%%%%%%
\begin{lem}[Step-size lower bound] 
\label{lem:lowBoundSS}
Let Assumption~\ref{ass:Holder} hold and let $\seq{x^k}$ be generated by Algorithm~\ref{alg:dealArmijo}. Then, the maximum number of  inner iterations  $p_k$ and the step-size $\alpha_k$ satisfy $p_k\leq \ov p$ and $\alpha_k\geq \tilde{\alpha}$, where
    \begin{equation}\label{eq:pkUpperAlphaLower}
        \ov p:= 1+\tfrac{\log \ov c}{\log \eta}, \quad \ov c:=\left(\tfrac{(1+\nu)(1-\sigma) c_1}{L \ov\alpha^{\nu} c_2^{1+\nu}}\right)^{\nicefrac{1}{\nu}}, \quad  \tilde{\alpha}:=\eta^{\ov p} \ov\alpha.
    \end{equation}
\end{lem}
%%%%%%%%%%
%%%%%%%%%%%%%
\begin{proof}
    Without loss of generality, let $p_k\ge 1$. By definition, $p_k$ is the smallest integer for which the Armijo rule holds, i.e., 
    \begin{equation}\label{eq:aniArmijo}
        f(x^k+\eta^{p_k-1}\ov\alpha d^k)> f(x^k)+\sigma \eta^{p_k-1}\ov\alpha \innprod{\nabla f(x^k)}{d^k}.
    \end{equation}
    The H\"{o}lderian descent lemma gives, for all $\alpha>0$,
    \[
    f(x^k+\alpha d^k) \le f(x^k)+\alpha \innprod{\nabla f(x^k)}{d^k}+\tfrac{L}{1+\nu} \alpha^{1+\nu} \|d^k\|^{1+\nu}.
    \]
    Substituting $\alpha=\eta^{p_k-1}\ov\alpha$ and combining with \eqref{eq:aniArmijo} yields
    \begin{align*}
        f(x^k)+\sigma \eta^{p_k-1}\ov\alpha \innprod{\nabla f(x^k)}{d^k} < f(x^k)+\eta^{p_k-1}\ov\alpha \innprod{\nabla f(x^k)}{d^k}+\tfrac{L}{1+\nu} (\eta^{p_k-1}\ov\alpha)^{1+\nu} \|d^k\|^{1+\nu}.
    \end{align*}
    Together with the descent condition \eqref{eq:genSufDescCond} and $\sigma\in (0,1)$, this implies
    \[
    (1-\sigma) c_1\|\nabla f(x^k)\|^{2+\beta} \leq (\sigma-1)  \innprod{\nabla f(x^k)}{d^k}
    < \tfrac{L}{1+\nu}(\eta^{p_k-1}\ov\alpha)^{\nu} \|d^k\|^{1+\nu}
        \leq \tfrac{L}{1+\nu} (\eta^{p_k-1}\ov\alpha)^{\nu} c_2^{1+\nu} \|\nabla f(x^k)\|^{(1+\nu)(1+\beta)}.
    \]
    With $\beta=\tfrac{1-\nu}{\nu}$, this reduces to
    \begin{align*}
        (1-\sigma) c_1 \leq \tfrac{L}{1+\nu} (\eta^{p_k-1}\ov\alpha)^{\nu} c_2^{1+\nu}, 
    \end{align*}
    leading to
    \begin{align*}
        \ov c:=\left(\tfrac{(1+\nu)(1-\sigma) c_1}{L \ov\alpha^{\nu} c_2^{1+\nu}}\right)^{\tfrac{1}{\nu}} \leq \eta^{p_k-1}.
    \end{align*}
    Hence $p_k\leq 1+\tfrac{\log \ov c}{\log \eta}=\ov p$. Since $\alpha_k=\eta^{p_k}\ov\alpha$, we also obtain $ \alpha_k\geq \eta^{\ov p} \ov\alpha$, adjusting our claims.
    \qed
\end{proof}

DEAL-A can be interpreted as an instance of DEAL by selecting appropriate parameters.
\begin{lem}[DEAL-A as an instance of DEAL]\label{lem:dealA:deal}
Let Assumption~\ref{ass:Holder} hold and let $\seq{x^k}$ be generated by Algorithm \ref{alg:dealArmijo}.
Then, this sequence can be reproduced by Algorithm~\ref{alg:deal}, by setting $\theta =\beta+2=1+\tfrac{1}{\nu}$
and  $\rho = \sigma \tilde{\alpha}c_1$, where $\tilde{\alpha}$ is given in \eqref{eq:pkUpperAlphaLower}.
\end{lem}
\begin{proof}
From the setting of Algorithm \ref{alg:dealArmijo}, we have $\theta>1$, $\rho>0$, and 
from \eqref{eq:DealArmio2}, $x^{k+1}$ is generated such that satisfies \eqref{eq:deal}.
\qed
\end{proof}

We note that DEAL-A inherits the full convergence and complexity guarantees of DEAL as stated below.
\begin{cor}[Convergence and complexity analysis of DEAL-A]\label{cor:conDEALA}
Let Assumption \ref{ass:Holder} hold and let\linebreak $\seq{x^k}$ be generated by Algorithm \ref{alg:dealArmijo}.
Then, the results of Theorems~\ref{thm:GlobConv-s},~\ref{thm:LinConvWeakDeal-s},~\ref{thm:GlobConv},~and~\ref{thm:LinConvWeakDealFun} and Corollary~\ref{cor:onconvinglob} hold
by setting $\theta =\beta+2=1+\tfrac{1}{\nu}$, $c = c_2 \ov{\alpha}$,
and  $\rho = \sigma \tilde{\alpha}c_1$, where $\tilde{\alpha}$ is introduced in \eqref{eq:pkUpperAlphaLower}.
\end{cor}
\begin{proof}
The result follows from Lemma~\ref{lem:dealA:deal} and Remark~\ref{rem:cononiter} with $\alpha_{\max}=\ov\alpha$. \qed
\end{proof}

%%%%%%%%%%%%%%%%%%%%%%%%%%%%%%%%%%%%%%%%%%%%%%%%%%%%%%%%%%%%%%%%%%%%%%%%%%%%%%%%%%
\subsection{{\bf DEAL in the Nonsmooth Setting}}\label{sub:fben} 
In DEAL and its two instances, DEAL-C and DEAL-A, the objective function $f$ in problem~\eqref{eq:prob} is assumed to be smooth. 
We now extend the framework to nonsmooth settings by considering a proper lsc function 
$\varphi: \R^n \to \overline{\R}$ and the problem
\begin{equation}\label{eq:nonprob}
\min_{x \in \mathbb{R}^n} \varphi(x).
\end{equation}
A prominent smoothing approach to handle this problem relies on 
the \textit{high-order proximal-point operator} and \textit{high-order Moreau envelope} (HOME) defined as
\[
\mathrm{prox}_{\gamma \varphi}^p (x):=\arg\min_{y\in \R^n} \left\{\varphi(y)+\tfrac{1}{p\gamma}\Vert x- y\Vert^p\right\},
\qquad \varphi_{\gamma}^p (x):=\min_{y\in \R^n} \left\{\varphi(y)+\tfrac{1}{p\gamma}\Vert x- y\Vert^p\right\},
\]
where $p>1$ is the order of the regularization term, and $\gamma>0$ is a proximal parameter. 
If $p=2$, these notions are known as the \textit{proximal operator} and the \textit{Moreau envelope} \cite{moreau1965proximite}.

The following fact collects basic properties essential for our analysis.

\begin{fact}\label{fact:basicprophome}\cite{Kabgani24itsopt,Kabganidiff}
Let $p>1$ and $\varphi: \mathbb{R}^n \to \overline{\mathbb{R}}$ be a proper lsc function and bounded below. The following assertions hold:
\begin{enumerate}
    \item $\mathrm{prox}_{\gamma \varphi}^p (x)$ is nonempty and compact and $\varphi_{\gamma}^p (x)$ is finite for every $x\in \R^n$ and $\gamma>0$;
    \item For each $\gamma>0$, $\inf_{y\in \R^n}  \varphi_{\gamma}^p(y)=\inf_{y\in \R^n} \varphi(y)$ and
    $\arg\min_{y\in \R^n} \varphi_{\gamma}^p(y)= \arg\min_{y\in \R^n} \varphi(y)$;
    \item for each open subset $U \subseteq \R^n$,  $\varphi_{\gamma}^p \in \mathcal{C}^{1}(U)$ if and only if 
    $\mathrm{prox}_{\gamma \varphi}^p$ is nonempty, single-valued, and continuous on $U$. Moreover, under these conditions, for any $x \in U$, we have 
    \begin{equation}\label{eq:gardhome}
    \nabla\varphi_{\gamma}^p(x) = \tfrac{1}{\gamma} \Vert x - \mathrm{prox}_{\gamma \varphi}^p(x) \Vert^{p-2} (x - \mathrm{prox}_{\gamma \varphi}^p(x)).
    \end{equation}
\end{enumerate}
\end{fact}

Differentiability properties of HOME under conditions such as 
\textit{$q$-prox-regularity}, \textit{$p$-calmness}, or weak convexity have been widely studied 
\cite{Kabganitechadaptive,Kabganidiff}. 
In the convex case, $\varphi_\gamma^p$ is globally differentiable for all $p>1$ and $\gamma>0$ \cite[Chapter~12]{bauschke2017convex}, 
while in the weakly convex case, differentiability requires a sufficiently small $\gamma$ 
(e.g., restricted to a ball containing the level set $\mathcal{L}(x^0)$) \cite{Kabganitechadaptive}. 
Related results under prox-regularity appear in \cite[Theorem~4.4]{poliquin1996proxregular}. 
To avoid technicalities, we simply assume $\varphi_\gamma^p \in \C^1$ whenever differentiability is needed. 

In the remainder of this subsection, we recall two algorithms built on these smoothing tools: 
the Boosted high-order proximal-point algorithm (Boosted HiPPA) for $p>1$, the boosted proximal gradient algorithm (Boosted PGA) for $p=2$. 
We show that all these methods fit the DEAL framework, and hence their convergence, rates, 
and complexity follow directly from the results in Section~\ref{sec:genericDescMethod}. 
To keep the exposition concise, we introduce HiPPA and PGA as instances of DEAL in 
Examples~\ref{ex:PPM}~and~\ref{ex:FBM}, while restricting detailed convergence and complexity analysis to their boosted versions.

%%%%%%%%%%%%%%%%%%%%%%%
\begin{exa}[High-order proximal-point algorithm]\label{ex:PPM}
Consider problem \eqref{eq:nonprob} where $\varphi$ is proper, lsc, and bounded below. Assume that for some $\gamma>0$, $\varphi_\gamma^p \in \C^1$. Hence, from Fact~\ref{fact:basicprophome}, $\mathrm{prox}_{\gamma \varphi}^p (x)$ is nonempty and single-valued for each $x\in \R^n$. The high-order proximal-point algorithm iterates via
\[
x^{k+1}=\mathrm{prox}_{\gamma \varphi}^p (x^k).
\]
The gradient of the HOME at $x^k$ obtains from \eqref{eq:gardhome}
and the update can be rewritten as
\[
x^{k+1}=\mathrm{prox}_{\gamma \varphi}^p (x^k)= x^k  - \gamma^{\frac{1}{p-1}}\left\Vert \nabla\varphi_{\gamma}^p (x^k)\right\Vert^{\frac{2-p}{p-1}}\nabla \varphi_{\gamma}^p(x^k).
\] 
Moreover, the definition of $\varphi_\gamma$ yields the descent inequality
\[
\varphi_{\gamma}^p(x^{k+1}) \leq \varphi(x^{k+1}) = \varphi_{\gamma}^p(x^k)-\tfrac{1}{p\gamma}\Vert x^{k+1}-x^k\Vert^p,
\]
which, combined with the gradient expression, yields
\begin{equation}\label{eq:ineq:hipp}
\varphi_{\gamma}^p(x^{k+1})\leq \varphi_{\gamma}^p(x^k)-\tfrac{\gamma^{\frac{1}{p-1}}}{p}\Vert \nabla\varphi_{\gamma}^p(x^k)\Vert^{\frac{p}{p-1}}.
\end{equation}
This shows that HiPPA fits the DEAL framework with $\rho=\tfrac{\gamma^{\frac{1}{p-1}}}{p}$ and
$\theta=\tfrac{p}{p-1}$.
For $p=2$, this reduces to the classical proximal-point algorithm with $\rho=\tfrac{\gamma}{2}$ and $\theta=2$, 
giving a concrete nonsmooth instance of DEAL.\qed
\end{exa}

To improve flexibility and acceleration, we introduce the \textit{Boosted HiPPA} \cite{Kabgani24itsopt}(Algorithm~\ref{alg:boostedhippa}), which augments the update with a correction term as
\[
x^{k+1}=(1-t_k)\mathrm{prox}_{\gamma \varphi}^p (x^k)+t_k \left(x^k+d^k\right),
\]
where $d^k\in\R^n$ is a search (not necessarily descent) direction, $\sigma\in \big(0, \tfrac{1}{p\gamma}\big)$, and $t_k=\eta^m$, with $\eta\in(0,1)$ and $m\in\N_0$.  
In addition, it holds that
\[
\varphi_{\gamma}^p(\mathrm{prox}_{\gamma \varphi}^p (x^k))\leq \varphi(\mathrm{prox}_{\gamma \varphi}^p (x^k))=\varphi_{\gamma}^p(x^k)-\tfrac{1}{p\gamma}\Vert x^k - \mathrm{prox}_{\gamma \varphi}^p (x^k)\Vert^p
<\varphi_{\gamma}^p(x^k)- \sigma\gamma^{\frac{p}{p-1}}\Vert \nabla\varphi_{\gamma}^p(x^k)\Vert^{\frac{p}{p-1}}.
\]
Since $\eta^m\downarrow 0$ as $m\to \infty$, it is evident that
\[
(1-\eta^m)\mathrm{prox}_{\gamma \varphi}^p (x^k)+\eta^m(x^k+d^k)\to \mathrm{prox}_{\gamma \varphi}^p (x^k).
\]
From the continuity of $\varphi_{\gamma}^p$ (see \cite[Theorem~3.4]{Kabgani24itsopt}), there exists some  $\ov{m}_k\in \N_0$ such that by setting $t_k=\eta^{\ov{m}_k}$,
\[
\varphi_{\gamma}^p\left((1-t_k)\mathrm{prox}_{\gamma \varphi}^p (x^k)+t_k \left(x^k+d^k\right)\right)
\leq \varphi_{\gamma}^p(x^k)- \sigma\gamma^{\frac{p}{p-1}}\Vert \nabla\varphi_{\gamma}^p(x^k)\Vert^{\frac{p}{p-1}}.
\]
This aligns Boosted HiPPA with DEAL, using $\rho=\sigma\gamma^{\frac{p}{p-1}}$ and $\theta=\frac{p}{p-1}$.

\vspace{4mm}
%%%%%%%%%%%%%%%%%%%%%%
\RestyleAlgo{boxruled}
\begin{algorithm}[H]
\DontPrintSemicolon
\textbf{Initialization:} $x^0\in\R^ n$,~ $\eta\in (0,1)$,~ $\gamma>0$,~ $\sigma\in \big(0, \tfrac{1}{p\gamma}\big)$,~ $k=0$;

    \Repeat{
        the stopping criterion holds
    }{
    Choose a direction $d^k\in \R^n$ and set $m=0$;\;
    Let $t_k\in\{\eta^m\mid m\in \N\}$ be the largest such that 
    $x^{k+1}=(1-t_k)\mathrm{prox}_{\gamma \varphi}^p (x^k)+t_k \left(x^k+d^k\right)$ satisfies
    \[
    \varphi_{\gamma}^p(x^{k+1})\leq \varphi_{\gamma}^p(x^k)-\sigma\gamma^{\frac{p}{p-1}}\Vert \nabla\varphi_{\gamma}^p(x^k)\Vert^{\frac{p}{p-1}};
    \]
    Set $k=k+1$;
    }

\caption{Boosted HiPPA (Boosted High-order proximal-point Algorithm) \label{alg:boostedhippa}}
\end{algorithm}
\vspace{4mm}

Next, we show that Boosted HiPPA can be cast as a special case of DEAL.
\begin{lem} [Boosted HiPPA as an instance of DEAL]\label{lem:BPPA:deal}
Let $\varphi_\gamma^p\in\C^1$ and $\seq{x^k}$ be generated by Algorithm \ref{alg:boostedhippa}.
Then, this sequence can be reproduced using Algorithm~\ref{alg:deal}, by setting $\rho=\sigma\gamma^{\frac{p}{p-1}}$ and $\theta=\frac{p}{p-1}$.
\end{lem}
\begin{proof}
From the setting of Algorithm \ref{alg:boostedhippa}, we have $\theta>1$, $\rho>0$, and 
from \eqref{eq:ineq:BPG}, $x^{k+1}$ is generated such that satisfies \eqref{eq:deal}.
\qed
\end{proof}

As a direct consequence, Boosted HiPPA inherits the full convergence and complexity guarantees of DEAL.
\begin{cor}[Convergence and complexity analysis of  Boosted HiPPA]\label{cor:conBPPA}
Let $\varphi$ satisfy the KL inequality at a reference point with exponent $\vartheta$, 
and assume $\varphi_\gamma^p\in\C^1$ with $p=\tfrac{1}{1-\vartheta}$. 
If $\seq{x^k}$ is generated by Algorithm~\ref{alg:boostedhippa}, then the results of 
Theorems~\ref{thm:GlobConv-s},~\ref{thm:LinConvWeakDeal-s},~\ref{thm:GlobConv},~and~\ref{thm:LinConvWeakDealFun} and Corollary~\ref{cor:onconvinglob} hold
by setting $\rho=\sigma\gamma^{\frac{p}{p-1}}$ and $\theta=\frac{p}{p-1}$.
\end{cor}
\begin{proof}
The claim follows from Lemma~\ref{lem:BPPA:deal} and \cite[Theorem~3.9]{Kabgani24itsopt}.\qed
\end{proof}

%%%%%%%%%%%%%%%%%%%%%%%
Let us consider the composite optimization problem 
\begin{equation}\label{eq:probfb}
    \min_{x\in\R^n}~\varphi(x):=f(x)+g(x),
\end{equation}
where $f:\R^n\to\R$ is continuously differentiable with $L$-Lipschitz continuous gradient and $g:\R^n\to\Rinf$ is a proper lsc function (possibly nonsmooth) that is bounded below. The forward-backward operator for $\varphi$ is defined as
\[
T_\gamma(x) := \arg\min_{y \in \mathbb{R}^n} \left\{ f(x) + \langle \nabla f(x), y - x \rangle + \tfrac{1}{2\gamma} \| y - x \|^2 + g(y) \right\}.
\]
It is straightforward that $T_\gamma(x) = \mathrm{prox}_{\gamma g}^2 (x - \gamma \nabla f(x))$. 
As discussed in Fact~\ref{fact:basicprophome}, for $p=2$, the operator $T_\gamma(x)$ is single-valued under suitable conditions on $g$, such as convexity or weak convexity, for an appropriate choice of $\gamma$. If $g$ is bounded below and $f$ is twice continuously differentiable, additional assumptions on $g$, such as ($\varrho$-weak) convexity \cite[Theorem~2.6]{stella2017forward} or prox-regularity \cite[Theorem~4.7]{themelis2018forward}, guarantee that the \textit{forward–backward envelope} (FBE) with parameter $\gamma>0$,  
\begin{equation}\label{eq:fbeq}
\varphi_\gamma(x) := \inf_{y \in \mathbb{R}^n} \left\{ f(x) + \langle \nabla f(x), y - x \rangle + \tfrac{1}{2\gamma} \| y - x \|^2 + g(y) \right\},
\end{equation}
is continuously differentiable (possibly for restricted ranges of $\gamma$). Its gradient is given by
\begin{equation}\label{eq:ex:FBM:nab}
\nabla \varphi_\gamma(x) = \tfrac{1}{\gamma} \left( I - \gamma \nabla^2 f(x) \right) \left( x - T_\gamma(x) \right),
\end{equation}
and moreover \cite[Proposition~4.3]{themelis2018forward},
\begin{equation}\label{eq:ex:FBM:ineq}
\varphi_\gamma(T_\gamma(x)) \leq \varphi(T_\gamma(x)) \leq \varphi_\gamma(x) - \tfrac{1 - \gamma L}{2\gamma} \| x - T_\gamma(x) \|^2.
\end{equation}
With these preliminaries, we now recall the proximal gradient algorithm.

\begin{exa}[Proximal gradient algorithm]\label{ex:FBM}
Consider problem \eqref{eq:probfb}. Let $f$ be twice differentiable with an $L$-Lipschitz continuous gradient and $g$ proper, lsc, and bounded below, 
and assume $\varphi_\gamma\in\C^1$. 
The proximal gradient method iterates
\[
x^{k+1} := T_\gamma(x^k)=\mathrm{prox}_{\gamma g}^2 \left( x^k - \gamma \nabla f(x^k) \right).
\]
Since $\Vert \nabla^2 f(x) \Vert \leq L$ \cite[Lemma~1.2.2]{nesterov2018lectures}, the eigenvalues $\lambda_i$ of $\nabla^2 f(x)$ lie in $[-L, L]$. Defining $A = I - \gamma \nabla^2 f(x)$, the eigenvalues of $A$ are $1 - \gamma \lambda_i \in [1 - \gamma L, 1 + \gamma L]$. By choosing $\gamma < \tfrac{1}{L}$, we ensure $1 - \gamma L > 0$, and thus \eqref{eq:ex:FBM:nab} leads to
\begin{equation}\label{eq:ex:FBM:a}
\| \nabla \varphi_\gamma(x) \| \leq \tfrac{1 + \gamma L}{\gamma} \| x - T_\gamma(x) \|.
\end{equation}
Together with \eqref{eq:ex:FBM:ineq}, this yields the descent inequality
\[
\varphi_\gamma(T_\gamma(x)) \leq \varphi_\gamma(x) - \tfrac{1 - \gamma L}{2\gamma} \| x - T_\gamma(x) \|^2 \leq \varphi_\gamma(x) - \tfrac{\gamma (1 - \gamma L)}{2(1 + \gamma L)^2} \| \nabla \varphi_\gamma(x) \|^2.
\]
Thus, PGA fits the DEAL framework with $\rho=\tfrac{\gamma (1 - \gamma L)}{2(1 + \gamma L)^2}$ and $\theta=2$.\qed
\end{exa}

As shown in Example~\ref{ex:FBM}, if $\varphi_\gamma$ given in \eqref{eq:fbeq} is continuously differentiable, then
\eqref{eq:ex:FBM:a} is valid.
Moreover, \cite[Lemma~5.1]{themelis2018forward} shows that by selecting a direction $d^k\in\R^n$ and $\gamma<\tfrac{1}{L}$, one can find a step-size $t_k>0$ such that with $x^{k+1}=T_\gamma(x^k)+t_k d^k$,
\[
\varphi_\gamma(x^{k+1})\leq \varphi_\gamma(x^k)-\tfrac{\sigma}{\gamma^2} \Vert x^k - T_\gamma(x^k)\Vert^2,
\]
for some $\sigma \in (0,\tfrac{\gamma(1-\gamma L)}{2})$. 
Hence, we have
\begin{equation}\label{eq:ineq:BPG}
   \varphi_\gamma(x^{k+1})\leq \varphi_\gamma(x^k)-\tfrac{\sigma}{\left(1+\gamma L\right)^2} \Vert\nabla \varphi_\gamma(x^k)\Vert^2, 
\end{equation}
which again corresponds to DEAL with $\rho=\tfrac{\sigma}{(1+\gamma L)^2}$ and $\theta=2$. 
This boosted variant, which can be viewed as a special instance of the ZeroFPR algorithm \cite{themelis2018forward}, is summarized in Algorithm~\ref{alg:zeroFPR}.

\vspace{4mm}
%%%%%%%%%%%%%%%%%%%%%%
\RestyleAlgo{boxruled}
\begin{algorithm}[H]
\DontPrintSemicolon
\textbf{Initialization:} $x^0\in\R^ n$,~ $\ov\alpha\in (0,1)$,~ $0<\gamma<\frac{1}{L}$,~ $\sigma\in \big(0, \tfrac{\gamma(1-\gamma L)}{2}\big)$,~ $k=0$;

    \Repeat{
        the stopping criterion holds
    }{
    Choose a direction $d^k\in \R^n$;\;
    Let $t_k\in\{\ov\alpha^m\mid m\in \N\}$ be the largest such that $x^{k+1}=T_\gamma(x^k)+t_k d^k$ satisfies
    \[
    \varphi_\gamma(x^{k+1})\leq \varphi_\gamma(x^k)-\tfrac{\sigma}{\left(1+\gamma L\right)^2} \Vert\nabla \varphi_\gamma(x^k)\Vert^2;
    \]
    Set $k=k+1$;
    }

\caption{Boosted PGA (Boosted Proximal Gradient Algorithm) \label{alg:zeroFPR}}
\end{algorithm}
\vspace{4mm}

Let us first show that Boosted PGA fits within the general DEAL framework through suitable choices of $\theta$ and $\rho$.
\begin{lem}[Boosted PGA as an instance of DEAL]\label{lem:dealBPG:deal}
Let $\varphi_\gamma\in\C^1$ and $\seq{x^k}$ be generated by Algorithm \ref{alg:zeroFPR}.
Then, this sequence can be reproduced by Algorithm~\ref{alg:deal}, by setting $\theta =2$
and $\rho=\frac{\sigma}{\left(1+\gamma L\right)^2}$.
\end{lem}
\begin{proof}
From the setting of Algorithm \ref{alg:zeroFPR}, we have $\theta>1$, $\rho>0$, and 
from \eqref{eq:ineq:BPG}, $x^{k+1}$ is generated such that satisfies \eqref{eq:deal}.
\qed
\end{proof}

Boosted PGA inherits the full convergence and complexity guarantees of DEAL as stated in the following.

\begin{cor}[Convergence and complexity analysis of Boosted PGA]\label{cor:conDEALBPG}
Let the function $f$ be twice continuously differentiable with $L$-Lipschitz gradient.
Let $\varphi$ satisfy the KL inequality at a reference point with exponent $\vartheta\in [\frac{1}{2}, 1)$ and
$\varphi_\gamma\in\C^1$ with $\gamma\in (0, L^{-1})$. If $\seq{x^k}$ is generated by Algorithm \ref{alg:zeroFPR}, then the results of Theorems~\ref{thm:GlobConv-s},~\ref{thm:LinConvWeakDeal-s},~\ref{thm:GlobConv},~and~\ref{thm:LinConvWeakDealFun} and Corollary~\ref{cor:onconvinglob} hold
by setting $\theta =2$
and $\rho=\frac{\sigma}{\left(1+\gamma L\right)^2}$.
\end{cor}
\begin{proof}
The claim follows directly from Lemma~\ref{lem:dealBPG:deal} and \cite[Remark~5.1(ii)]{yu2022kurdyka}.\qed
\end{proof}
A key advantage of Boosted HiPPA lies in its flexibility: by tuning $p$, it achieves linear convergence for any KL exponent $\vartheta \in (0,1)$. 
Indeed, if $\varphi_{\gamma}^p$ satisfies the KL inequality with exponent $\vartheta$, choosing $\theta=\frac{1}{\vartheta}=\frac{p}{p-1}$, i.e., $p=\frac{1}{1-\vartheta}$, guarantees global and linear convergence. 
This adaptability contrasts with methods such as Boosted PGA, whose linear convergence is restricted to the case $\vartheta=\tfrac{1}{2}$.

% %%%%%%%%%%%%%%%%%%%%%%%%%%%%%%%%%%%%%%%%%%%%%%%%%%%%%%%%%%%%%%%%%%%%%%%%%%%%%%%%%%
% %%%%%%%%%%%%%%%%%%%%%%%%%%%%%%%%%%%%%%%%%%%%%%%%%%%%%%%%%%%%%%%%%%%%%%%%%%%%%%%%%%
\section{Preliminary Numerical Experiments} \label{sec:preNumExper}
This section reports numerical experiments illustrating the behavior of the proposed generalized descent
framework on representative smooth and smoothed nonsmooth objectives satisfying the Kurdyka-{\L}ojasiewicz
(KL) property. We first describe the computational environment and stopping criteria, then introduce the test
problems, and finally discuss the observed performance.
%%%%%%%%%%%%%%%%%%%%%%%%%%%%%%%%%%%%%%%%%%%%%%%%%%%%%%%%%%%%%%%%%%%%%%%%%%%%%%%%%%

All numerical experiments were conducted on a MacBook Pro with an Apple M1 processor and 16GB RAM. The algorithms were implemented in Python 3.9 using NumPy and SciPy libraries. The termination criterion was set as \(\|\nabla f(x^k)\| \leq \varepsilon\) with \(\varepsilon = 10^{-6}\), and the maximum number of iterations was capped at 10,000.
The step-size parameters for the DEAL algorithm were selected as follows:
\begin{itemize}
    \item[$\bullet$] \textbf{Constant step-size:} \(\alpha_k = \alpha\) with \(\alpha\) chosen based on theoretical guarantees.
    \item[$\bullet$] \textbf{Armijo line search:} The initial step-size \(\overline{\alpha} = 1\), reduction factor \(\eta = 0.5\), and acceptance parameter \(\sigma = 10^{-4}\).
\end{itemize}
For all methods, initial points were randomly chosen from a uniform distribution over \([-5,5]^n\), unless otherwise specified.

% %%%%%%%%%%%%%%%%%%%%%%%%%%%%%%%%%%%%%%%%%%%%%%%%%%%%%%%%%%%%%%%%%%%%%%%%%%%%%%%%%%
\subsection{{\bf Robust Phase Retrieval}}
Let $\{a_i\}_{i=1}^m \subseteq \mathbb{R}^n$ be the sensing vectors, and let
$\{b_i\}_{i=1}^m \subseteq \mathbb{R}$ be the phaseless quadratic measurements.
We consider the problem of recovering an unknown signal $x^* \in \mathbb{R}^n$
from observations of the form
\[
b_i = \langle a_i, x^* \rangle^2, \qquad i=1,\dots,m.
\]
Such measurement models naturally arise in phase retrieval, where only the magnitudes of the linear measurements are available.

To this end, we study the following nonconvex optimization problem:
\begin{equation}\label{eq:rpr_obj}
\min_{x\in\R^n} f(x)
\;:=\;
\frac{1}{\delta}\sum_{i=1}^m
\bigl|\langle a_i , x\rangle^2 - b_i\bigr|^{\delta},
\end{equation}
where $\delta\in [1,2]$. 
The objective function $f$ is continuously differentiable and belongs to the
class $\C^{1,\nu}_{L_R}(B(0;R))$ for every $R>0$, with
$\nu = \delta-1$.
The following theorem further establishes that $f$ satisfies the
KL property in a neighborhood of the optimal solution
$x^*$.
Moreover, the associated KL exponent is given by
$\vartheta = 1 + \frac{1}{\nu}$,
which constitutes a sufficient condition for the linear convergence of
\textsc{DEAL-C} and \textsc{DEAL-A}; see Corollaries~\ref{cor:conDEALc}
and~\ref{cor:conDEALA}, as well as Theorem~\ref{thm:LinConvWeakDeal-s}. We recall that, although the case $\delta=1$ is also of practical interest in robust phase retrieval, it renders the objective function in~\eqref{eq:rpr_obj} nonsmooth. For this reason, we do not consider this case in the present experiments.
Nevertheless, the nonsmooth setting can be addressed within the proposed framework by the methods introduced in Subsection~\ref{sub:fben}.

\begin{prop}\label{prop:phase:HKL}
Let $f$ be defined by \eqref{eq:rpr_obj}. Let us assume that there exists a (global) minimizer $x^*$ satisfying the conditions
\begin{equation}\label{eq:consistency}
\langle a_i , x^*\rangle^2 = b_i,\qquad  i=1,\dots,m,
\end{equation}
and 
\begin{equation}\label{eq:nondeg}
G(x^*):=\big[\nabla q_1(x^*),\dots,\nabla q_m(x^*)\big]\in\mathbb{R}^{n\times m}
\quad\text{satisfies}\quad
\sigma_{\min}\big(G(x^*)\big)>0,
\end{equation}
where 
$q_i(x):=\langle a_i , x\rangle^2 - b_i$ and $\sigma_{\min}\big(G(x^*)\big)$ is the smallest singular value of the matrix $\big(G(x^*)\big)$.
Then the following statements hold:
\begin{enumerate}
\item \label{prop:phase:HKL:a} For every $R>0$, there exists $L_R>0$ such that
\begin{equation}\label{eq:holder_grad}
\|\nabla f(x)-\nabla f(y)\|\;\le\;L_R\|x-y\|^{\delta-1},
\qquad \forall x,y\in B(0;R).
\end{equation}
In particular, $f\in \C^{1,\nu}_{L_R}(B(0;R))$ on bounded sets with H\"older exponent $\nu=\delta-1$.

\item \label{prop:phase:HKL:b} There exist $r>0$ and $c>0$ such that
\begin{equation}\label{eq:KL_standard}
\|\nabla f(x)\|\;\ge\;c\bigl(f(x)-f(x^*)\bigr)^{\vartheta},
\qquad \forall x\in B(x^*;r),
\end{equation}
with the KL exponent $\vartheta=\frac{\delta-1}{\delta}$.

\item \label{prop:phase:HKL:c} The exponents in \eqref{eq:holder_grad}~and~\eqref{eq:KL_standard} satisfy
\begin{equation}\label{eq:identity}
\frac{1}{\vartheta}=1+\frac{1}{\nu}.
\end{equation}
\end{enumerate}
\end{prop}
\begin{proof}
    The proof is deferred to Appendix~\ref{App}.
\end{proof}

Note that Proposition~\ref{prop:phase:HKL}, along with Theorem~\ref{thm:GlobConv-s}, Remark~\ref{rem:cononiter}, and Corollary~\ref{cor:conDEALc} (resp. Corollary~\ref{cor:conDEALA}), guarantees that the sequence generated by DEAL-C (resp. DEAL-A) for solving the problem~\eqref{eq:rpr_obj} converges to the global solution $x^*$. 

In the numerical experiments, we consider a noise-free phase retrieval setting.
The ambient dimension of the unknown signal is denoted by $n$, and the number
of measurements by $m$, with $m \gg n$.
Unless otherwise stated, we fix $n=200$ and $m=5000$.
The sensing vectors $\{a_i\}_{i=1}^m$ are generated
independently from the standard Gaussian distribution, i.e.,
$a_i \sim \mathcal{N}(0,I_n)$, and are subsequently normalized to have the unit
Euclidean norm.
The ground-truth signal $x^* \in \mathbb{R}^n$ is generated independently from
$\mathcal{N}(0,I_n)$ and normalized so that $\|x^*\|_2 = 1$.
The phaseless quadratic measurements are then constructed according to
\[
b_i = \langle a_i, x^* \rangle^2, \qquad i=1,\dots,m,
\]
which ensures consistency of the data with the model \eqref{eq:rpr_obj}.
Under this construction, $x^*$ is a global minimizer of $f$ and satisfies
$q_i(x^*)=0$ for all $i$.
Moreover, since $m \gg n$, the nondegeneracy condition
$\sigma_{\min}(G(x^*))>0$ holds with high probability.
We consider the case $\delta=\frac{3}{2}$. Since Proposition~\ref{prop:phase:HKL} establishes the KL inequality only in a neighborhood of a global minimizer $x^*$, the purpose of this experiment is to illustrate the corresponding local linear convergence regime. Accordingly, the algorithm is initialized near the ground-truth solution by
\[
x^0 = x^* + 0.5\,\frac{\xi}{\Vert \xi\Vert},
\qquad
\xi \sim \mathcal{N}(0,I_n),
\]
where $\xi$ is a standard Gaussian random vector independent of $x^*$.

We implement both DEAL-C and DEAL-A for this problem. For DEAL-C, since an explicit H\"older constant of $\nabla f$ is not available in closed form for this instance, we choose the fixed step-size $\alpha$ empirically by a short warm-up search from $x_0$: starting from $\alpha_0=1$, we repeatedly halve $\alpha$ until the corresponding fixed-step trajectory is monotone over the first $K=30$ trial iterations. The selected $\alpha$ is then used in the full run, with the additional safeguard that the run is terminated if a function increase is detected. Thus, the monotonicity reported for DEAL-C is verified empirically along the computed trajectory, rather than guaranteed \textit{a priori} from an explicit H\"older constant.
For DEAL-A, we set the Armijo parameter $\sigma = 10^{-4}$, the backtracking factor $\eta = 0.8$, and the initial trial step-size $\bar{\alpha} = 1$.

Figure~\ref{fig:DEAL:ph}(a) reports the empirical ratios
\[
R_f(k):=\frac{f_{k+1}}{f_k},
\qquad
R_g(k):=\frac{\|\nabla f_{k+1}\|}{\|\nabla f_k\|},
\qquad
R_i(k):=\frac{\|x^{k+1}-x^*\|}{\|x^k-x^*\|},
\]
generated by DEAL-C. All ratios stabilize strictly below one, providing clear numerical evidence of linear convergence of the function values, gradient norms, and iterates, in agreement with the theoretical predictions under the KL property. 
DEAL-A, Figure~\ref{fig:DEAL:ph}(b) and Table~\ref{tab:dealA} show that the empirical ratios remain uniformly below~$1$, thus confirming the linear convergence. However, these ratios are consistently larger than those observed for DEAL-C, indicating a slower linear contraction rate for DEAL-A in this setting.

\begin{table}
\centering
\begin{tabular}{|c|c|c|c|c|c|c|}
\hline 
Ratio & 0 & 100 & 300 & 500 & 700 & 900 \\ \hline
$R_f$ & 9.092e-01 & 9.812e-01 & 9.777e-01 & 9.775e-01 & 9.775e-01 & 9.775e-01 \\
\hline
$R_g$ & 5.754e-01 & 9.925e-01 & 9.924e-01 & 9.924e-01 & 9.924e-01 & 9.924e-01 \\
\hline
$R_i$ & 9.404e-01 & 9.880e-01 & 9.851e-01 & 9.849e-01 & 9.849e-01 & 9.849e-01 \\
\hline
\end{tabular}\caption{ Empirical linear contraction ratios $R_f$, $R_g$, and $R_i$ for DEAL-A, reported at iterations $k = 0, 100, 300, 500, 700,$ and $900$. The values remain uniformly below~$1$, corroborating linear convergence, while their magnitude reflects a slower contraction rate compared to \textsc{DEAL-C}.}\label{tab:dealA}
\end{table}

\begin{figure}[!htbp]
    \centering
        \begin{subfigure}[t]{0.49\linewidth}
        \centering
        \includegraphics[width=\linewidth]{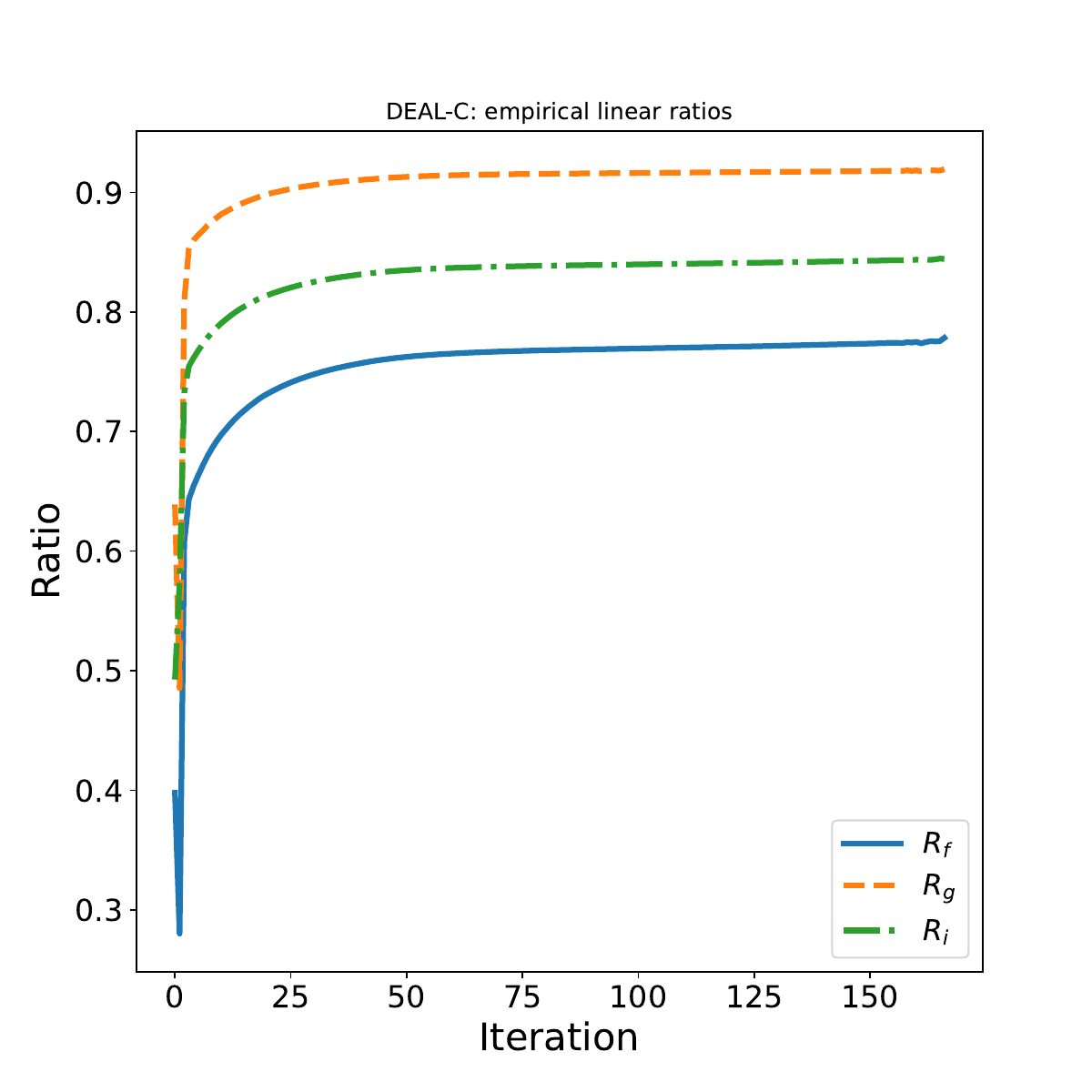}
        \caption{Empirical contraction ratios for DEAL-C.}
        \label{fig:cameraman_loss}
    \end{subfigure}
    \begin{subfigure}[t]{0.49\linewidth}
        \centering
        \includegraphics[width=\linewidth]{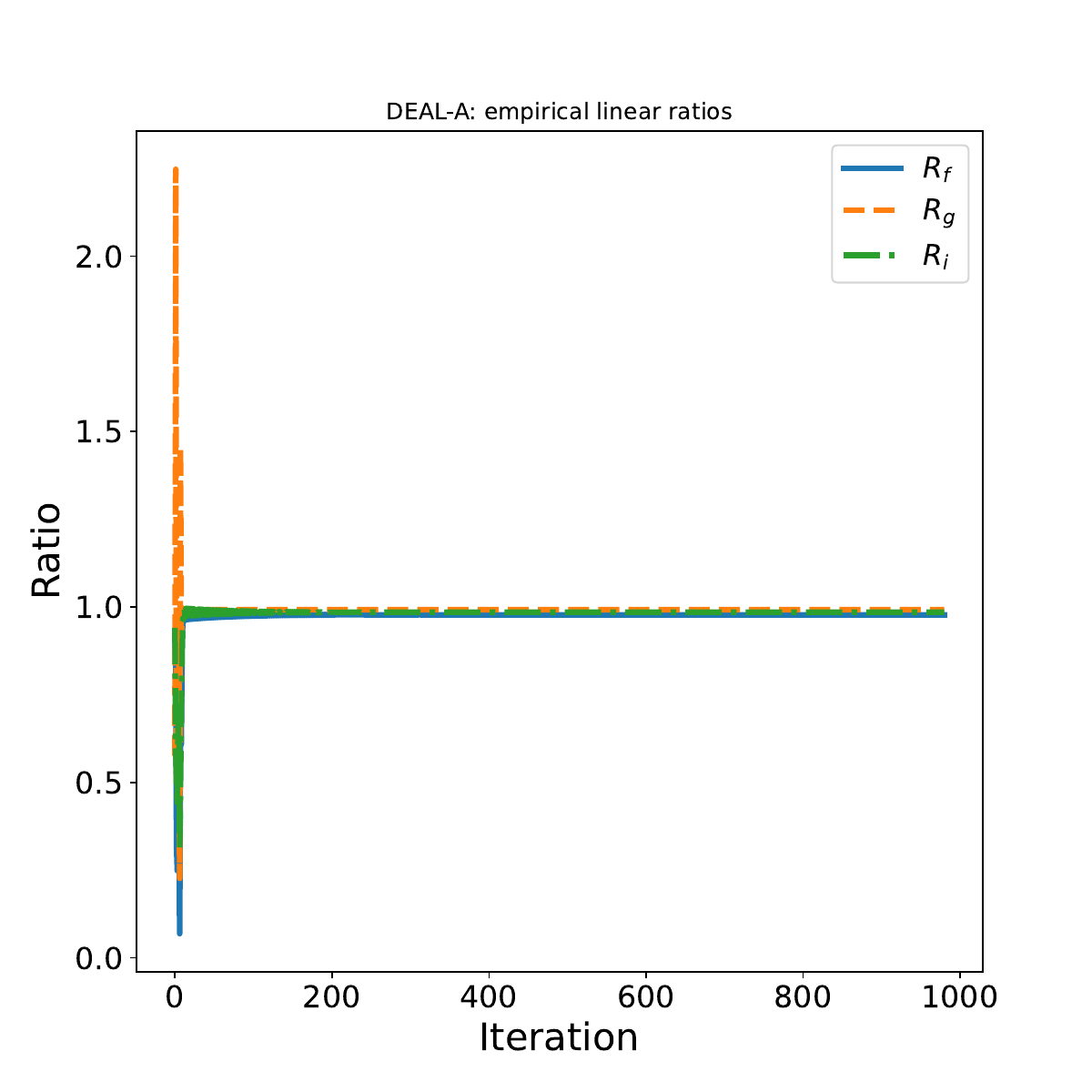}
        \caption{Empirical contraction ratios for DEAL-A.}
        \label{fig:cameraman_psnr}
    \end{subfigure}\hfill
    \caption{Empirical contraction ratios for the noise-free robust phase retrieval problem with $p=3/2$.
Subfigure~(a) corresponds to DEAL-C, 
while subfigure~(b) corresponds to DEAL-A.
In both cases, the ratios 
$R_f$, 
$R_g$, 
and $R_i$ remain strictly below~$1$ after a short transient phase, 
providing numerical evidence of linear convergence.
Compared with DEAL-C, DEAL-A exhibits larger contraction ratios, 
reflecting a slower linear convergence rate in this example. }
    \label{fig:DEAL:ph}
\end{figure}

% %%%%%%%%%%%%%%%%%%%%%%%%%%%%%%%%%%%%%%%%%%%%%%%%%%%%%%%%%%%%%%%%%%%%%%%%%%%%%%%%%%
\subsection{{\bf Linear Inverse Problem with Nonconvex Regularization}}
In numerous scientific and engineering disciplines, such as signal and image processing, machine learning, compressed sensing, geophysics, and statistics, one often encounters scenarios in which the underlying features of interest cannot be directly observed or measured. Instead, these features must be inferred indirectly through the observed data. Such problems are collectively known as \textit{inverse problems}. A particularly important subclass is that of \textit{linear inverse problems}, which arise when a linear relationship exists between the unknown variables and the observed data. Specifically, if \( b \in \mathbb{R}^m \) denotes an observation of an unknown signal \( x \in \mathbb{R}^n \), and \( A : \mathbb{R}^m \rightarrow \mathbb{R}^n \) is a known linear operator, the canonical formulation of the linear inverse problem is given by
\begin{equation}\label{eq:linInvProb1}
 b = Ax + e,
\end{equation}
where \( e \in \mathbb{R}^m \) represents additive or impulsive noise, typically modeled with limited or qualitative prior knowledge. Due to the ill-conditioning of the linear system and the presence of noise, direct inversion is usually infeasible or unstable. Consequently, one often resorts to regularized optimization techniques to obtain stable approximations of the solution. A common approach involves solving the generalized least-$\delta$ norm optimization problem defined as
\begin{equation}\label{eq:GLP-R}
    \min_{x\in\R^n}  f(x)=\tfrac{1}{\delta}\|Ax-b\|^\delta+\lambda \mathcal{R}(x).
\end{equation}
where \(A \in \mathbb{R}^{m \times n}\) is a full rank matrix with the smallest singular value $\sigma_{\min}>0$, \(b \in \mathbb{R}^m\), and \(\delta\in (1,2]\). Furthermore, \(\func{\mathcal{R}}{\mathbb{R}^n}{\mathbb{R}}\) is a sparsity-promoting regularizer, and \(\lambda > 0\) controls the trade-off between data fidelity and sparsity. Although convex regularizers such as (e.g. \(\ell_1\)- or \(\ell_2\)-norms) are more commonly considered in literature, nonconvex sparsity-promoting penalties \(\mathcal{R}\) have been shown to improve recovery accuracy; see, e.g. \cite{chen2014convergence,liu2018robust,yang2019weakly}. Motivated by this consideration, we focus on \(\mathcal{R}(x) = \sum_{i=1}^n r(x_i)\) in which the specific penalty \(\func{r}{\mathbb{R}}{\mathbb{R}}\) is defined by
\[
r(t) = 
\begin{cases}
2|t|^\delta - \delta t^{2},\hspace*{4mm} & \text{if}\hspace*{2mm} |t| \le 1,\\[2mm]
2-\delta, \hspace*{4mm} & \text{if}\hspace*{2mm} |t| > 1,
\end{cases}
\]
with $\delta\in (1,2)$. It is obvious that this function is a weakly convex function on $\R$ and is even, non-decreasing on \([0, \infty)\), not identically zero with \(r(0)=0\), and \(t \mapsto \tfrac{r(t)}{t}\) is non-increasing on \((0, \infty)\). 

Since the objective function $f$ in \eqref{eq:GLP-R} is semialgebraic, it satisfies the KL inequality. Moreover, by the following proposition, the function possesses a  $(\delta-1)$-H\"{o}lder gradient (the proof is left in the Appendix~\ref{App}). 

\begin{prop}\label{prop:invprob1}
   Let  $f$ be the function defined by \eqref{eq:GLP-R}. Then, the function $f$ has a $(\delta-1)$-H\"{o}lder gradient with constant $2^{2-\delta}\|A\|^{\delta}+2^{4-\delta} n^{1-\tfrac{\delta}{2}}p\lambda$.
\end{prop}

We next illustrate the numerical behavior of the proposed methods on the inverse problem~\eqref{eq:GLP-R}, focusing on image deblurring with additive Gaussian noise; see, e.g., \cite{jackie2007deblurring}. All experiments are conducted with
\[
\delta=\tfrac{5}{3}, \qquad \lambda = 10^{-2}, \qquad p= \tfrac{5}{3},
\]
for which the objective function
admits a \((\delta-1)=\tfrac{2}{3}\)-H\"older continuous gradient; see Proposition~\ref{prop:invprob1}.

The forward operator $A$ models spatially invariant blur and is implemented as a two-dimensional convolution with a normalized $5 \times 5$ averaging kernel using FFTs under periodic boundary conditions. Observations are generated according to~\eqref{eq:linInvProb1}, where $x^\star$ denotes the standard \textit{Cameraman} image normalized to $[0,1]$, and $e$ is zero-mean white Gaussian noise with variance chosen to achieve a blurred signal-to-noise ratio (BSNR) of $40$ dB. In this image deblurring setting, both the unknown image and the observation have the same dimensions; therefore, the initial iterate is chosen as the observed blurred/noisy image, i.e., $x^0 = b$.
In our comparison, we consider the following algorithms:
\begin{itemize}
\item[$\bullet$] \textbf{DEAL-A}: Algorithm~\ref{alg:dealConstantSS};
\item[$\bullet$] \textbf{DEAL-C}: Algorithm~\ref{alg:dealArmijo};
\item[$\bullet$] \textbf{BHiPPA~(Boosted HiPPA)}: Algorithm~\ref{alg:boostedhippa};
\item[$\bullet$] \textbf{BPGA~(Boosted PGA)}: Algorithm~\ref{alg:zeroFPR}.
\end{itemize}
All methods run for at most $300$ outer iterations under identical stopping criteria.
Performance is evaluated in terms of the objective value $f(x)$, while reconstruction quality is assessed using the peak signal-to-noise ratio (PSNR), computed with respect to the ground-truth image $x^\star$ as
\[
\mathrm{PSNR}(x,x^\star)=10\log_{10}\!\left(\frac{\mathrm{MAX}^2}{\mathrm{MSE}(x,x^\star)}\right),
\qquad
\mathrm{MSE}(x,x^\star)=\frac{1}{N^2}\,\|x-x^\star\|_F^2,
\]
where $N\times N$ is the image size and $\mathrm{MAX}=1$ since all images are normalized to $[0,1]$.

As Figure~\ref{fig:cameraman_curves}(a) shows, the loss function values exhibit a clear separation among the considered methods. While all methods work very well on the problem, DEAL-A and BHiPPA achieve the fastest decrease and stabilize earlier.
In terms of reconstruction accuracy, as in Figure~\ref{fig:cameraman_curves}(b) illustrated, BHiPPA attains a final PSNR of \(29.37\) dB, which is comparable to DEAL-A (\(29.59\) dB) and improves upon BPGA (\(28.62\) dB) and DEAL-C (\(27.14\) dB). Visual inspection of the reconstructions supports these observations.

\begin{figure}[!htbp]
    \centering
        \begin{subfigure}[t]{0.49\linewidth}
        \centering
        \includegraphics[width=\linewidth]{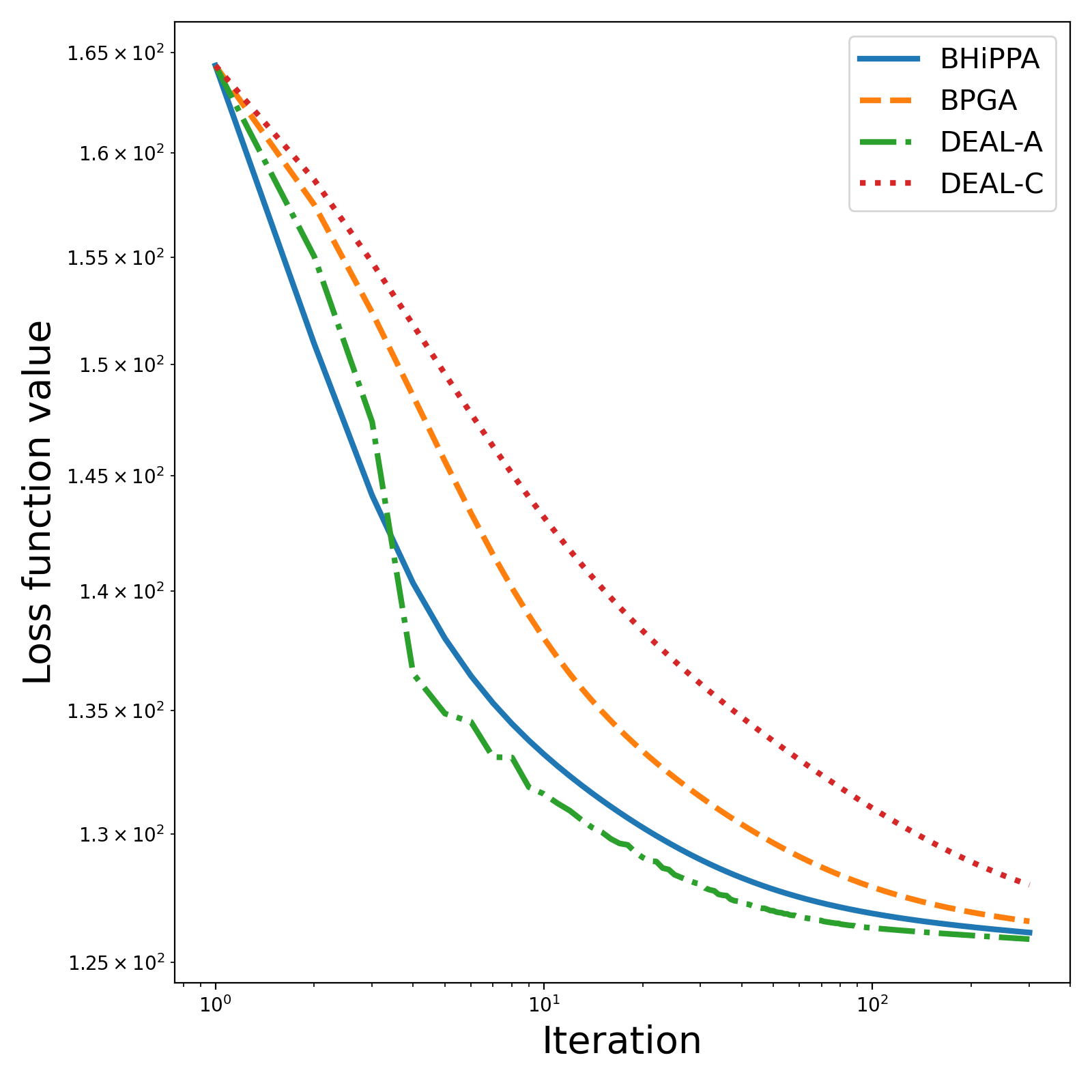}
        \caption{Loss value versus iteration (log scale).}
        \label{fig:cameraman_loss}
    \end{subfigure}
    \begin{subfigure}[t]{0.49\linewidth}
        \centering
        \includegraphics[width=\linewidth]{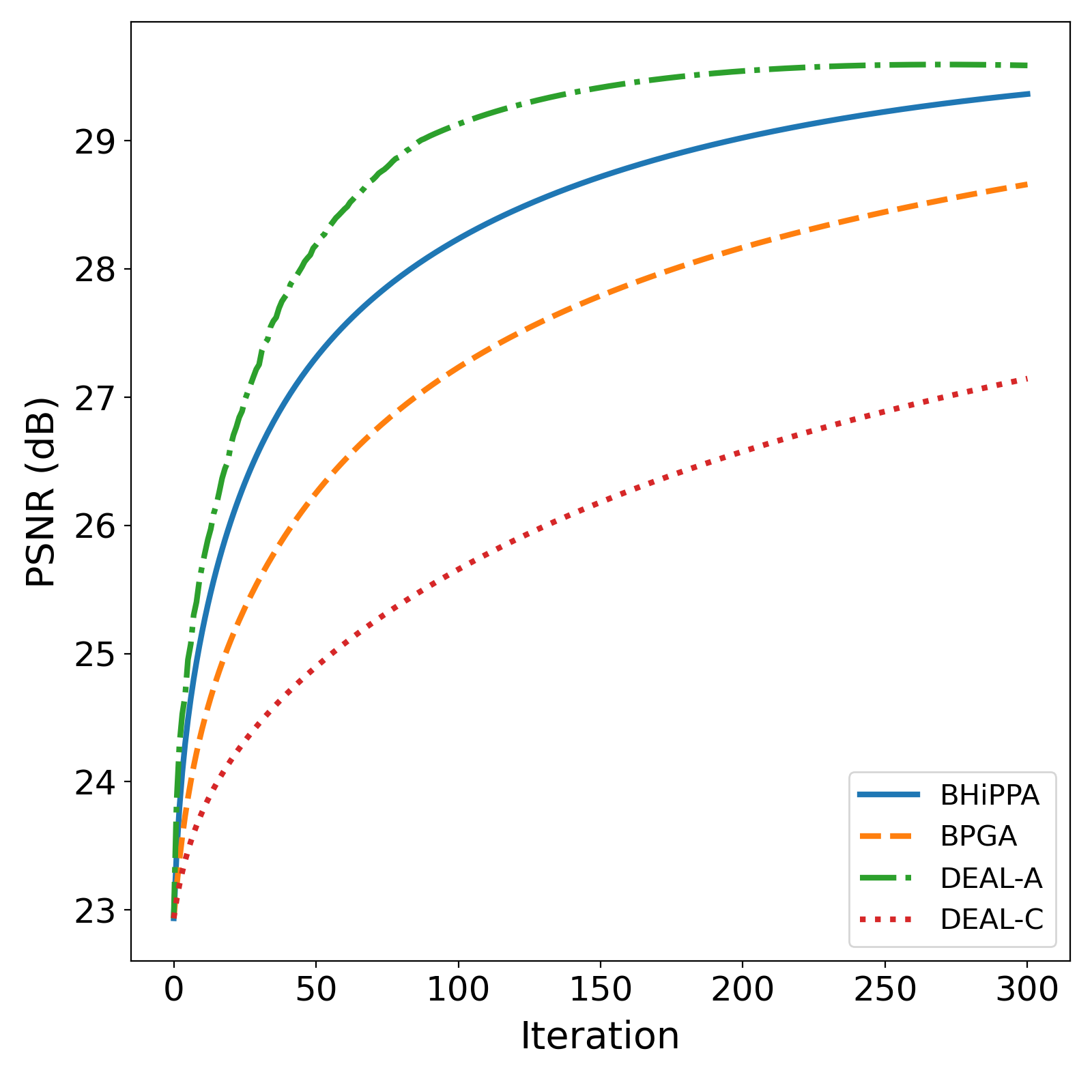}
        \caption{PSNR (dB) versus iteration.}
        \label{fig:cameraman_psnr}
    \end{subfigure}\hfill
    \caption{Cameraman deblurring.
(a) Objective value versus iteration (log scale).
(b) PSNR (dB) versus iteration.
PSNR is computed with respect to the ground-truth image after clipping reconstructions to $[0,1]$.
Comparison between BHiPPA, BPGA, DEAL-A, and DEAL-C.}
    \label{fig:cameraman_curves}
\end{figure}

Figure~\ref{fig:cameraman_recons} shows the reconstruction results for the Cameraman deblurring problem. Boosted HiPPA and DEAL-A provide the most accurate reconstructions, achieving PSNR values of $29.37$~dB and $29.59$~dB, respectively. BPGA yields a slightly degraded solution ($28.62$~dB), while DEAL-C produces an over-smoothed reconstruction with lower accuracy ($27.14$~dB). These results are consistent with the trends observed in Fig.~\ref{fig:cameraman_curves}.

\begin{figure}[!htbp]
    \centering
    \begin{subfigure}[t]{0.32\linewidth}
        \centering
        \includegraphics[width=\linewidth]{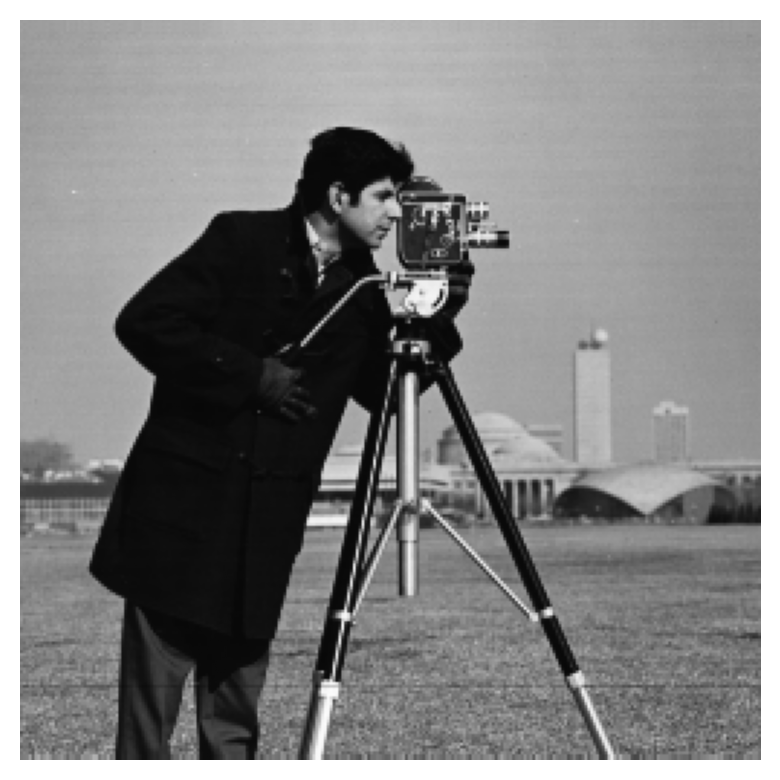}
        \caption{Ground truth.}
        \label{fig:cameraman_true}
    \end{subfigure}\hfill
    \begin{subfigure}[t]{0.32\linewidth}
        \centering
        \includegraphics[width=\linewidth]{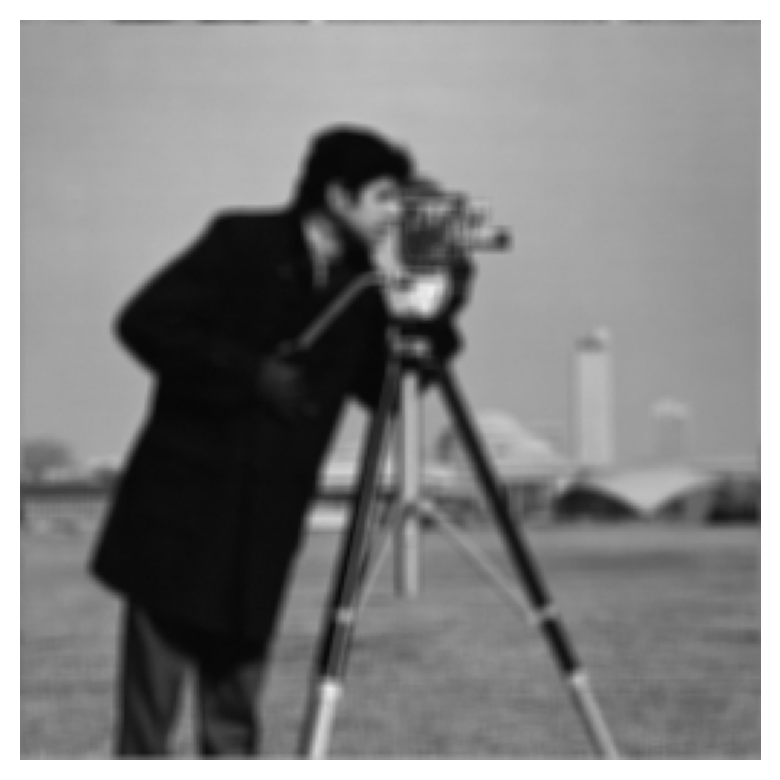}
        \caption{Blurred/noisy observation $y$.}
        \label{fig:cameraman_noisy}
    \end{subfigure}\hfill
    \begin{subfigure}[t]{0.32\linewidth}
        \centering
        \includegraphics[width=\linewidth]{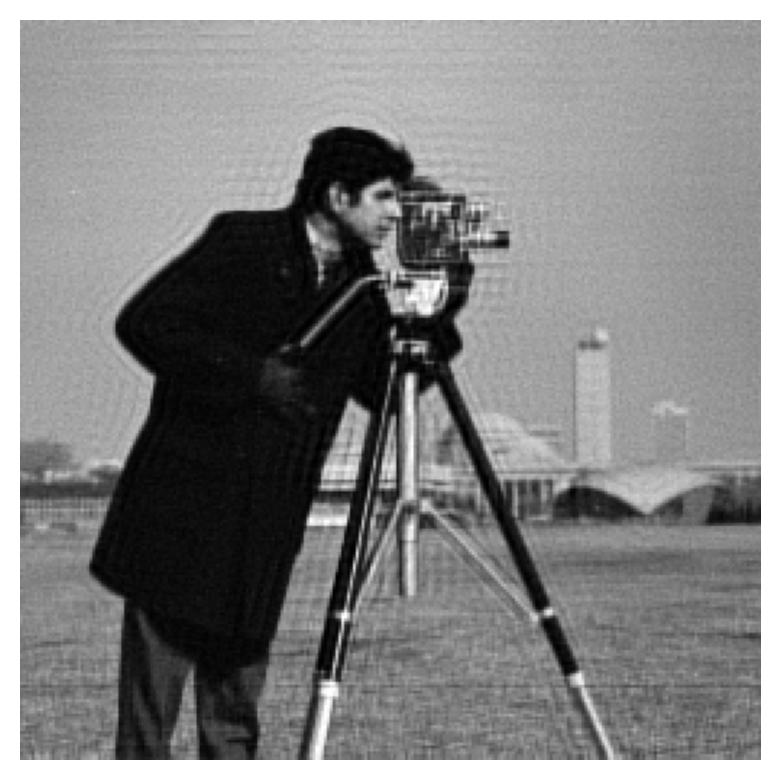}
        \caption{Boosted HiPPA (PSNR $=29.37$ dB).}
        \label{fig:cameraman_high}
    \end{subfigure}
    \vspace{0.6em}
    \begin{subfigure}[t]{0.32\linewidth}
        \centering
        \includegraphics[width=\linewidth]{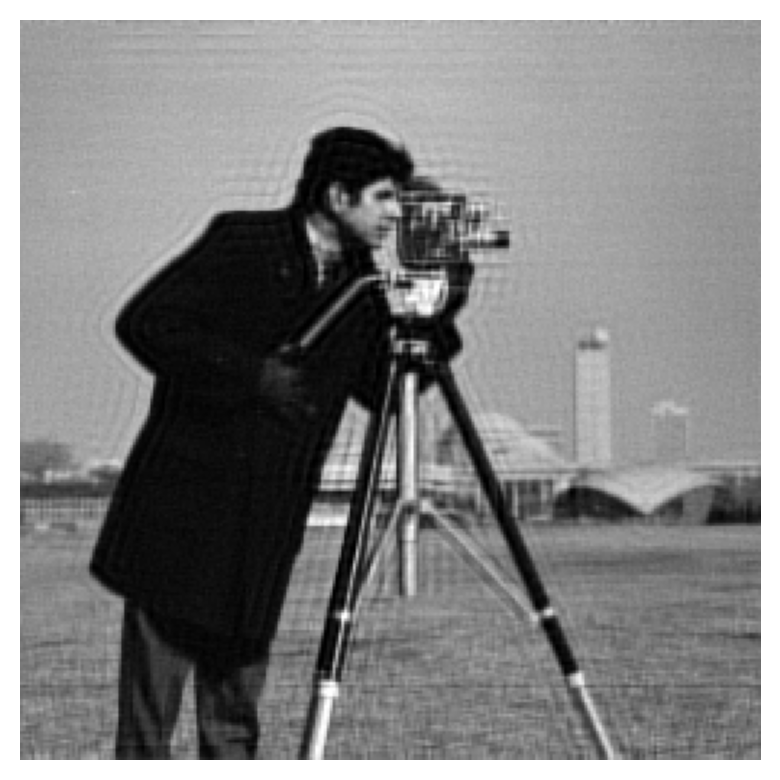}
        \caption{BPGA (PSNR $=28.62$ dB).}
        \label{fig:cameraman_bpga}
    \end{subfigure}\hfill
    \begin{subfigure}[t]{0.32\linewidth}
        \centering
        \includegraphics[width=\linewidth]{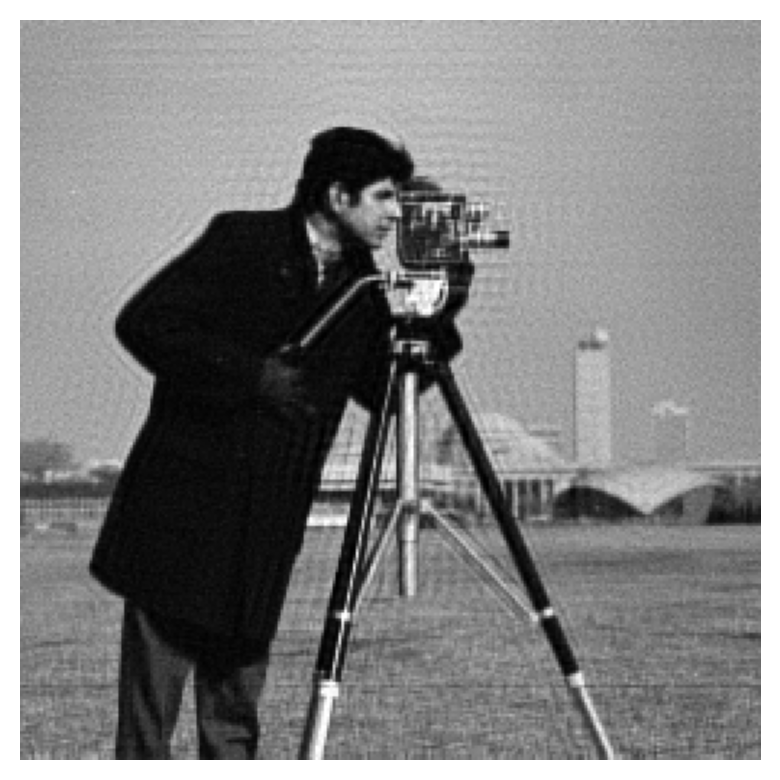}
        \caption{DEAL-A (PSNR $=29.59$ dB).}
        \label{fig:cameraman_deala}
    \end{subfigure}\hfill
    \begin{subfigure}[t]{0.32\linewidth}
        \centering
        \includegraphics[width=\linewidth]{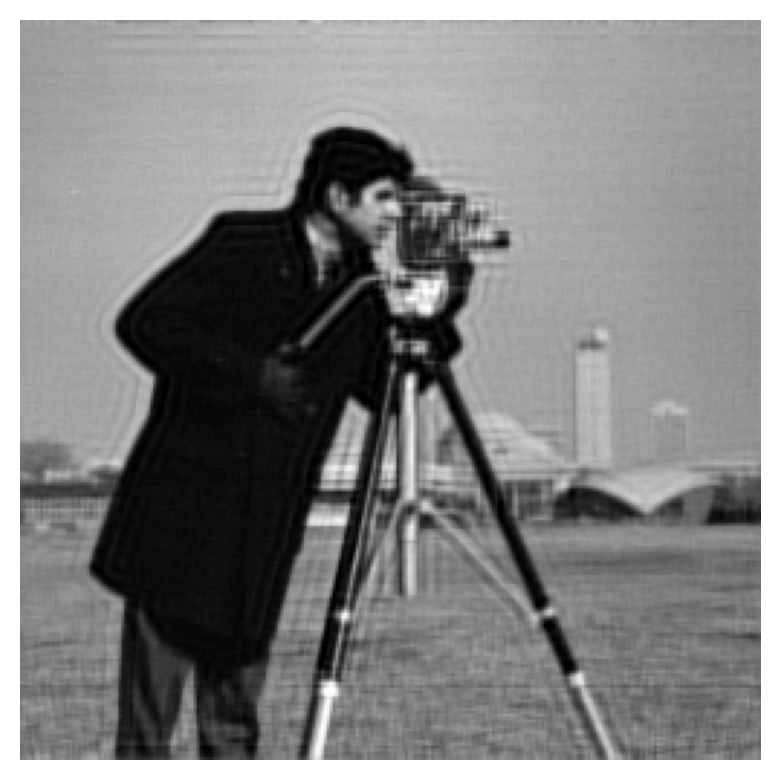}
        \caption{DEAL-C (PSNR $=27.14$ dB).}
        \label{fig:cameraman_dealc}
    \end{subfigure}
    \vspace{-3mm}
    \caption{Cameraman deblurring results.
Top row: (a) ground-truth image, (b) blurred/noisy observation, (c) Boosted HiPPA reconstruction.
Bottom row: (d) BPGA reconstruction, (e) DEAL-A reconstruction, (f) DEAL-C reconstruction.
All reconstructions are clipped to $[0,1]$ for visualization; PSNR values are reported in dB.}
    \label{fig:cameraman_recons}
\end{figure}

% %%%%%%%%%%%%%%%%%%%%%%%%%%%%%%%%%%%%%%%%%%%%%%%%%%%%%%%%%%%%%%%%%%%%%%%%%%%%%%%%%%

\subsection{{\bf Regularized Matrix Completion}}
% ---- Matrix completion (numerics) ----

Matrix completion aims to infer the unobserved entries of a partially observed matrix by exploiting an underlying low-rank structure. This setting arises, for instance, in collaborative filtering, where only a small subset of user-item ratings is available, and the full rating matrix is assumed to admit a low-dimensional latent representation.s
Given a partially observed matrix $A \in \mathbb{R}^{m \times n}$ and an index set
$\Omega \subseteq \{1,\dots,m\}\times\{1,\dots,n\}$, we consider the $\ell_1$-regularized low-rank factorization problem
\begin{equation}\label{eq:mc_l1}
    \min_{U \in \mathbb{R}^{m \times r},\, V \in \mathbb{R}^{n \times r}}
    \; \frac{1}{2} \big\| \mathcal{P}_\Omega(A-UV^\top) \big\|_F^2
    \;+\; \lambda \big( \|U\|_{1} + \|V\|_{1} \big),
\end{equation}
where $\mathcal{P}_\Omega:\mathbb{R}^{m\times n}\to\mathbb{R}^{m\times n}$ is the sampling (masking) operator defined entrywise by
\[
(\mathcal{P}_\Omega(X))_{ij} =
\begin{cases}
X_{ij} & (i,j)\in\Omega,\\
0 & (i,j)\notin\Omega,
\end{cases}
\]
equivalently, $\mathcal{P}_\Omega(X)=M\odot X$ with a binary mask $M\in\{0,1\}^{m\times n}$ such that $M_{ij}=1$ if $(i,j)\in\Omega$ and $M_{ij}=0$ otherwise.
Here, $\|\cdot\|_1$ is the entrywise $\ell_1$-norm and $\lambda>0$ is a regularization parameter.
The objective in~\eqref{eq:mc_l1} is semialgebraic and therefore satisfies the Kurdyka--\L{}ojasiewicz property locally at all points of its domain \cite{bolte2014proximal}.

We consider the MovieLens-100k dataset, consisting of $100{,}000$ ratings arranged in a sparse matrix $A \in \mathbb{R}^{943 \times 1682}$.
The observed entries are randomly split into a training set $\Omega_{\mathrm{tr}}$ (80\%) and a test set $\Omega_{\mathrm{te}}$ (20\%).
The rank parameter $r$ is fixed throughout, and all methods are initialized from the same starting point, where $U^{0}\in\mathbb{R}^{m\times r}$ and $V^{0}\in\mathbb{R}^{n\times r}$ are drawn from $\mathcal{N}(0,1)$.
The parameters $\beta$ and $p$ are selected through preliminary experiments; in view of Figures~\ref{fig:matrix_com}(a) and~\ref{fig:matrix_com}(b), we set $\beta=1$ and $p=1.25$ in all subsequent experiments. In our comparison, we used the following methods:
\begin{itemize}
\item[$\bullet$] \textbf{SGA-G}: Subgradient method with a geometrically decaying step-sizes $\alpha_k=\alpha_0\rho^k$~\cite{davis2018subgradient,rahimi2024projected} for some $0<\rho<1$;
\item[$\bullet$] \textbf{SGA-D}: Subgradient method with diminishing step-sizes, $\alpha_k=\alpha_0/\sqrt{k+1}$~\cite{davis2018subgradient,rahimi2024projected}; 
\item[$\bullet$]\textbf{ADAM}: An adaptive first-order method of \cite{kingma2014ADAM};
\item[$\bullet$]\textbf{BHiPPA~(Boosted HiPPA)}: Algorithm~\ref{alg:boostedhippa}.
\end{itemize}

Figure~\ref{fig:matrix_com}(b) reports the evolution of the training objective value versus iteration.
BHiPPA exhibits a faster decrease during the initial iterations and reaches lower objective values than the competing methods.
In contrast, SGA-G stagnates early due to aggressive step-size decay, while SGA-D and ADAM continue to decrease the objective but at a slower rate and to higher values.
Overall, these results indicate that the high-order envelope mechanism underlying BHiPPA provides more effective descent directions for nonsmooth matrix completion than direct subgradient-based updates.

%%%%%%%%%%%%%%%%%%%%%%

\begin{figure}[!htbp]
    \centering
        \begin{subfigure}[t]{0.49\linewidth}
        \centering
        \includegraphics[width=\linewidth]{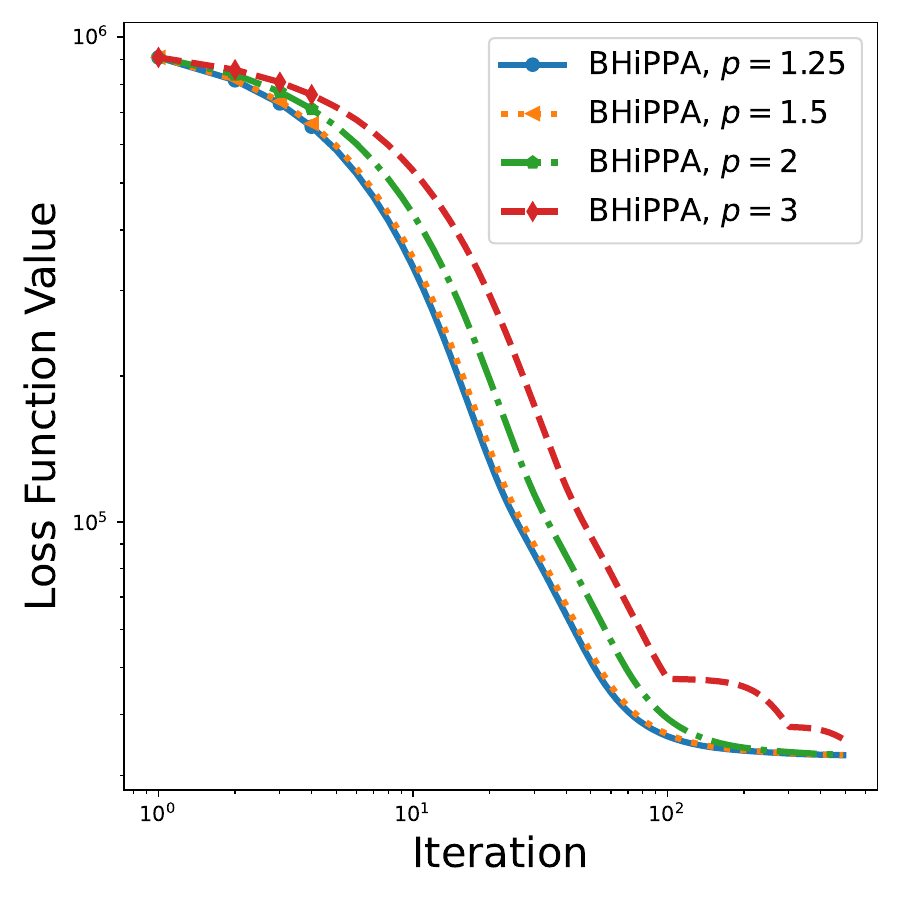}
        \caption{Loss value versus iteration (log scale).}
        \label{fig:cameraman_loss}
    \end{subfigure}
    \begin{subfigure}[t]{0.49\linewidth}
        \centering
        \includegraphics[width=\linewidth]{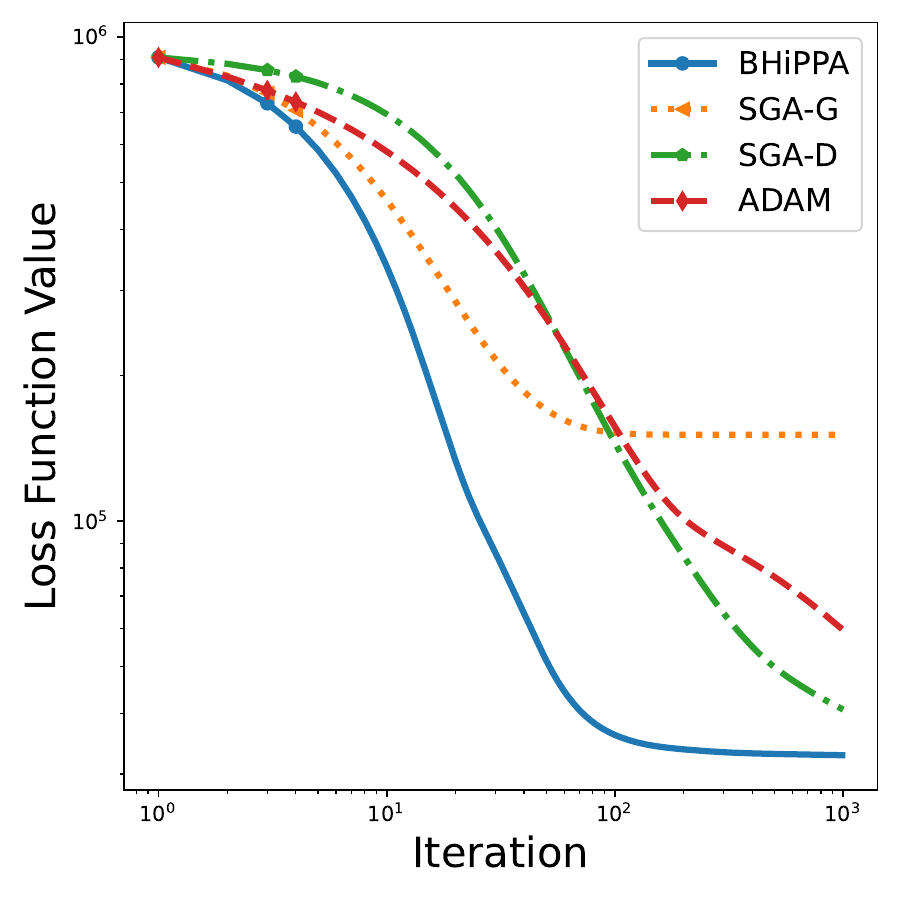}
        \caption{PSNR (dB) versus iteration.}
        \label{fig:cameraman_psnr}
    \end{subfigure}\hfill
    \caption{Regularized matrix completion (MovieLens-100k).
(a) Effect of the smoothing order $p$ on BHiPPA (objective value versus iteration, log scale).
(b) Objective value versus iteration (log scale) for BHiPPA, SGA-G, SGA-D, and ADAM.}
    \label{fig:matrix_com}
\end{figure}

\vspace{-5mm}
%%%%%%%%%%%%%%%%%%%%%%%%%%%%%%%%%%%%%%%%%%%%%%%%%%%%%%%%%%%%%%%%%%%%%%%%%%%%%%%%%%
%%%%%%%%%%%%%%%%%%%%%%%%%%%%%%%%%%%%%%%%%%%%%%%%%%%%%%%%%%%%%%%%%%%%%%%%%%%%%%%%%%
\section{Concluding Remarks} \label{sec:conclusion}
We developed a generalized descent framework (DEAL) to minimize smooth objective functions that satisfy the Kurdyka–{\L}ojasiewicz property. The proposed framework unifies several classical descent schemes and provides a transparent analysis of global convergence, convergence rates, and complexity under mild assumptions. In particular, we showed that by appropriately selecting the order of high-order proximal regularization, boosted proximal-point methods attain global linear convergence even for problems with arbitrary KL exponents. Numerical experiments on smooth inverse problems, convex nonsmooth models such as LASSO, and nonconvex matrix completion support the theoretical results and highlight the practical effectiveness of DEAL-type algorithms and BHiPPA-type algorithms.

\appendix
\section{Appendix} \label{App}
 We provide here the proof of our intermediate results for problems \eqref{eq:rpr_obj} and \eqref{eq:GLP-R}.\\[2mm]
%
%%%%%%%%%%%%%%%%%%%%%%%%%%%%%%%%%%
\textbf{ Proof of Proposition~\ref{prop:phase:HKL}}.
\ref{prop:phase:HKL:a}
Let $\phi(t)=\frac1\delta|t|^\delta$ be with $\delta\in(1,2]$. Then $\phi\in \C^1(\mathbb{R})$ and $\phi'(t)=|t|^{\delta-1}\operatorname{sign}(t)$ with
$\operatorname{sign}(0)=0$. 
Since $q_i$ $(i=1, \ldotp, m$) is polynomial, the function $f(x)=\sum_{i=1}^m \phi(q_i(x))$ belongs to $\C^1(\mathbb{R}^n)$. By the chain rule,
\begin{equation}\label{eq:grad_f}
\nabla f(x) =\sum_{i=1}^m \phi'(q_i(x))\,\nabla q_i(x),
\end{equation}
where $\nabla q_i(x)=2(a_i^\top x)a_i$. Moreover, for $\delta\in(1,2]$, $\phi$ has $\nu$-\textit{H\"older continuous gradient}  on $\mathbb{R}$
with exponent $\delta-1$ \cite[Theorem 6.3]{rodomanov2020smoothness}. In particular, there exists $C_\delta>0$ such that
\begin{equation}\label{eq:phi_holder}
|\phi'(u)-\phi'(v)|\le C_\delta |u-v|^{\delta-1}, \qquad \forall u,v\in\mathbb{R}.
\end{equation}
Fix $R>0$ and let $x,y\in B(0;R)$. Define $s_i(x):=\phi'(q_i(x))$.
From \eqref{eq:grad_f},
\[
\nabla f(x)-\nabla f(y)
=\sum_{i=1}^m (s_i(x)-s_i(y))\nabla q_i(x)
+\sum_{i=1}^m s_i(y)\bigl(\nabla q_i(x)-\nabla q_i(y)\bigr).
\]
Using the identity
\[
q_i(x)-q_i(y)
=(a_i^\top x)^2-(a_i^\top y)^2
=(a_i^\top(x-y))(a_i^\top(x+y)),
\]
and applying $\|x+y\|\le 2R$, we come up with
\begin{equation}\label{eq:qi_diff}
|q_i(x)-q_i(y)|
\le 2R\,\|a_i\|^2\,\|x-y\|.
\end{equation}
Combining \eqref{eq:phi_holder} and \eqref{eq:qi_diff} yields
\begin{equation}\label{eq:si_diff}
|s_i(x)-s_i(y)|
\le C_\delta(2R)^{\delta-1}\|a_i\|^{2(\delta-1)}\|x-y\|^{\delta-1}.
\end{equation}
On the other hand, for all $x\in B(0;R)$,
\begin{equation}\label{eq:grad_qi_bound}
\|\nabla q_i(x)\|\le 2R\|a_i\|^2,
\end{equation}
and
\begin{equation}\label{eq:grad_qi_lip}
\|\nabla q_i(x)-\nabla q_i(y)\|
\le 2\|a_i\|^2\|x-y\|.
\end{equation}
Invoking \eqref{eq:si_diff} and \eqref{eq:grad_qi_bound} leads to
\[
\|(s_i(x)-s_i(y))\nabla q_i(x)\|
\le 2C_\delta(2R)^{\delta-1}R\,\|a_i\|^{2\delta}\|x-y\|^{\delta-1}.
\]
Moreover, since $|s_i(y)|=|q_i(y)|^{\delta-1}$ and $|q_i(y)|\le \|a_i\|^2R^2+|b_i|$ on $B(0;R)$, there exists $M_{i,R}>0$ such that
\[
\|s_i(y)(\nabla q_i(x)-\nabla q_i(y))\|
\le 2M_{i,R}\|a_i\|^2\|x-y\|.
\]
Thanks to $\delta-1>0$ and $\|x-y\|\le 2R$, it holds that $\|x-y\|\le (2R)^{2-\delta}\|x-y\|^{\delta-1}$.
Summing over $i$ yields
\[
\|\nabla f(x)-\nabla f(y)\|
\le L_R\|x-y\|^{\delta-1},
\quad \forall x,y\in B(0;R),
\]
for some constant $L_R>0$. 
\\
\ref{prop:phase:HKL:b} From assumption, $q_i(x^*)=0$ for all $i$. Since $\nabla^2 q_i(x)=2a_i a_i^\top$, by choosing $h=x-x^*$, a Taylor expansion yields
\begin{equation}\label{eq:qi_exact_expansion}
q_i(x^*+h)=\nabla q_i(x^*)^\top h + (a_i^\top h)^2.
\end{equation}
Setting $r_i(h)=(a_i^\top h)^2$ for each $i$,  $q(x):=(q_1(x),\dots,q_m(x))^\top$, and $G(x):=[\nabla q_1(x),\dots,\nabla q_m(x)]$, this implies
\begin{equation}\label{eq:q_vector_form}
q(x^*+h)=G(x^*)^\top h + r(h).
\end{equation}
There exists a constant $c_1>0$ such that $\|r(h)\|_2\le c_1\|h\|^2$.
Using the reverse triangle inequality and the singular value bound,
we obtain
\begin{align*}
\|q(x^*+h)\|
&= \|G(x^*)^\top h + r(h)\| \ge \|G(x^*)^\top h\| - \|r(h)\| \\
&\ge \sigma_{\min}\big(G(x^*)\big)\|h\| - c_1\|h\|^2 = \big(\sigma_{\min}(G(x^*)) - c_1\|h\|\big)\|h\|.
\end{align*}
Choosing
\[
r := \min\left\{1,\ \frac{\sigma_{\min}(G(x^*))}{2c_1}\right\},
\]
for all $\|h\| \le r$, leads to
\begin{equation}\label{eq:lower_growth}
\frac12\,\sigma_{\min}\big(G(x^*)\big)\,\|h\|
\le \|q(x^*+h)\|.
\end{equation}
On the other hand, applying the triangle inequality to \eqref{eq:q_vector_form}
gives
\begin{align*}
\|q(x^*+h)\|
&\le \|G(x^*)^\top h\| + \|r(h)\|\le \|G(x^*)\|\,\|h\| + c_1\|h\|^2.
\end{align*}
For $\|h\| \le r \le 1$, we have $\|h\|^2 \le \|h\|$, and hence
\begin{equation}\label{eq:upper_growth}
\|q(x^*+h)\|
\le c_2 \|h\|,
\end{equation}
where $c_2 := \|G(x^*)\| + c_1$.
Combining \eqref{eq:lower_growth} and \eqref{eq:upper_growth}, we conclude
that there exists $r>0$ such that for all $\|h\| \le r$,
\begin{equation}\label{eq:q_linear_growth}
\frac12\,\sigma_{\min}\big(G(x^*)\big)\,\|h\|
\le \|q(x^*+h)\|
\le c_2 \|h\|.
\end{equation}
By defining $w_i(x):=|q_i(x)|^{\delta-1}\operatorname{sign}(q_i(x))$ and 
$w(x):=(w_1(x),\dots,w_m(x))^\top$, \eqref{eq:grad_f} can be rewritten as
\begin{equation}\label{eq:grad_matrix}
\nabla f(x)=G(x)w(x).
\end{equation}
The continuity of $G(\cdot)$ and $\sigma_{\min}(G(x^*))>0$, ensures that there exist $r>0$ and $\underline\sigma>0$ such that
$\sigma_{\min}(G(x))\ge \underline\sigma$ for all $x\in B(x^*;r)$. Hence,
\begin{equation}\label{eq:grad_lower_bound}
\|\nabla f(x)\|\ge \underline\sigma\|w(x)\|.
\end{equation}
Moreover, it is evident that
\[
f(x)=\frac1\delta\sum_{i=1}^m |q_i(x)|^\delta=\frac1\delta\|q(x)\|_\delta^\delta,
\qquad
\|w(x)\|_2=\Big(\sum_{i=1}^m |q_i(x)|^{2(\delta-1)}\Big)^{1/2}=\|q(x)\|_{2(\delta-1)}^{\delta-1}.
\]
Since $2(\delta-1)<\delta$, the monotonicity of $\ell$-norms yields $\|q(x)\|_\delta\le \|q(x)\|_{2(\delta-1)}$. Therefore,
\[
f(x)^{(\delta-1)/\delta}
=\delta^{-(\delta-1)/\delta}\|q(x)\|_\delta^{\delta-1}
\le \delta^{-(\delta-1)/\delta}\|q(x)\|_{2(\delta-1)}^{\delta-1}
=\delta^{-(\delta-1)/\delta}\|w(x)\|.
\]
Together with \eqref{eq:grad_lower_bound}, this implies
\[
\|\nabla f(x)\|
\ge \underline\sigma\,\delta^{(\delta-1)/\delta} f(x)^{(\delta-1)/\delta},
\qquad \forall x\in B(x^*;r).
\]
Since $f(x^*)=0$, this proves \ref{prop:phase:HKL:b} with $\vartheta=(\delta-1)/\delta$.
\\
Assertion~\ref{prop:phase:HKL:c} is concluded from Assertions~\ref{prop:phase:HKL:a}~and~\ref{prop:phase:HKL:b}.\qed
$~$\\[3mm]
\textbf{Proof of Proposition~\ref{prop:invprob1}.}
    By Proposition \ref{prop:invprob}, the function $\psi(x)=\tfrac{1}{\delta}\|Ax-b\|^{\delta}$ has a $(\delta-1)$-H\"{o}lder gradient with constant $2^{2-\delta}\|A\|^\delta$.
    Therefore, it suffices to show that the function $\mathcal{R}$ has a $(\delta-1)$-H\"{o}lder gradient with constant $2^{4-\delta} n^{1-\nicefrac{\delta}{2}}\delta$.
    We first indicate that the derivative of the scalar function $r$ is $(\delta-1)$-H\"older with constant $2^{4-\delta}\delta$. Let $t_1,t_2\in \R$. Three possible cases arise: (i) $|t_1|\le 1$ and $|t_2|\le 1$; (ii) $|t_1|\le 1$ and $|t_2|> 1$; (iii) $|t_1|> 1$ and $|t_2|> 1$.\\
    Let us consider Case (i) for which we have
    \begin{align*}
        |r'(t_1) - r'(t_2)| &=\Big| 2\delta|t_1|^{\delta-2}t_1 - 2\delta t_1 -2\delta|t_2|^{\delta-2}t_2 + 2\delta t_2\Big| \le 2\delta\Big| |t_1|^{\delta -2}t_1 -|t_2|^{\delta -2}t_2 \Big| + 2\delta |t_1-t_2|\\
        &\le 2^{4-\delta}\delta |t_1-t_2|^{\delta-1}.
    \end{align*}
    In Case (ii), $|t_1|\le 1$ and $|t_2|> 1$, it holds that $r'(t_2)=0=r'(1)$. Therefore, based on Case (i), we obtain
    \[
    |r'(t_1) - r'(t_2)| = |r'(t_1) - r'(1)| \le 2^{4-\delta}\delta |t_1-1|^{\delta-1}
    \le 2^{4-\delta}\delta |t_1-t_2|^{\delta-1}.
    \]
    For Case (iii), $r'(t_1)=r'(t_2)=0$ and the inequality holds trivially.\\
    Now, let $x,y\in \R^n$. We obtain
    \begin{align*}
        \|\nabla\mathcal{R}(x)-\nabla \mathcal{R}(y)\|^2 = \sum_{i=1}^n |r'(x_i) - r'(y_i)|^2\le
        2^{2(4-\delta)}\delta^2 \sum_{i=1}^n|x_i-y_i|^{2(\delta-1)}\le 2^{2(4-\delta)}\delta^2 n^{2-\delta} \|x-y\|^{2(\delta-1)},
    \end{align*}
    completing the proof.\qed
%%%%%%%%%%%

\addcontentsline{toc}{section}{References}
\bibliographystyle{spmpsci}      % mathematics and physical sciences
\bibliography{Bibliography}

@article{absil2005convergence,
	author		= {Absil, Pierre-Antoine and Mahony, Robert and Andrews, Ben},
	title		= {Convergence of the Iterates of Descent Methods for Analytic Cost Functions},
	journal		= {SIAM Journal on Optimization},
	volume		= {16},
	number		= {2},
	pages		= {531-547},
	year		= {2005}
}

@article{agarwal2021theory,
  title={On the theory of policy gradient methods: Optimality, approximation, and distribution shift},
  author={Agarwal, Alekh and Kakade, Sham M and Lee, Jason D and Mahajan, Gaurav},
  journal={Journal of Machine Learning Research},
  volume={22},
  number={98},
  pages={1--76},
  year={2021}
}

@article{ahookhosh2021bregman,
	  title={A {B}regman forward-backward linesearch algorithm for nonconvex composite optimization: Superlinear convergence to nonisolated local minima},
	  author={Ahookhosh, Masoud and Themelis, Andreas and Patrinos, Panagiotis},
	  journal={SIAM Journal on Optimization},
	  volume		= {31},
	  number		= {1},
	  pages		= {653--685},
	  year={2021}
}

@article{ahookhosh2019accelerated,
	title		= {Accelerated first-order methods for large-scale convex optimization: Nearly optimal complexity under strong convexity},
	author		= {Ahookhosh, Masoud},
	journal		= {Mathematical Methods of Operations Research},
	volume		= {89},
	number		= {3},
	pages		= {319--353},
	year		= {2019},
	publisher	= {Springer}
}

@InProceedings{allen2019convergence,
  title = 	 {A {C}onvergence {T}heory for {D}eep {L}earning via {O}ver-{P}arameterization},
  author =       {Allen-Zhu, Zeyuan and Li, Yuanzhi and Song, Zhao},
  booktitle = 	 {Proceedings of the 36th International Conference on Machine Learning},
  pages = 	 {242--252},
  year = 	 {2019},
  editor = 	 {Chaudhuri, Kamalika and Salakhutdinov, Ruslan},
  volume = 	 {97},
  series = 	 {Proceedings of Machine Learning Research},
  publisher =    {PMLR}
}

@article{apidopoulos2022convergence,
  title={Convergence rates for the heavy-ball continuous dynamics for non-convex optimization, under {P}olyak--{{\L}}ojasiewicz condition},
  author={Apidopoulos, Vassilis and Ginatta, Nicolo and Villa, Silvia},
  journal={Journal of Global Optimization},
  volume={84},
  number={3},
  pages={563--589},
  year={2022},
  publisher={Springer}
}

@article{aragon2018accelerating,
  title={Accelerating the {DC} algorithm for smooth functions},
  author={Arag{\'o}n Artacho, Francisco J and Fleming, Ronan MT and Vuong, Phan T},
  journal={Mathematical Programming},
  volume={169},
  pages={95--118},
  year={2018},
  publisher={Springer}
}

@article{attouch2010proximal,
	author		= {Attouch, H{\'e}dy and Bolte, J{\'e}r{\^o}me and Redont, Patrick and Soubeyran, Antoine},
	title		= {Proximal alternating minimization and projection methods for nonconvex problems: An approach based on the {K}urdyka-{{\L}}o\-ja\-sie\-wicz inequality},
	journal		= {Mathematics of Operations Research},
	volume		= {35},
	number		= {2},
	pages		= {438--457},
	year		= {2010},
	publisher	= {INFORMS}
}

@article{attouch2013convergence,
	author		= {Attouch, Hedy and Bolte, J\'er\^ome and Svaiter, Benar Fux},
	title		= {Convergence of descent methods for semi-algebraic and tame problems: Proximal algorithms, forward-backward splitting, and regularized {G}auss-{S}eidel methods},
	journal		= {Mathematical Programming},
	year		= {2013},
	month		= {Feb},
	day			= {01},
	volume		= {137},
	number		= {1},
	pages		= {91--129},
	issn		= {1436-4646},
	nodoi		= {10.1007/s10107-011-0484-9},
}

@book{bauschke2017convex,
	author		= {Bauschke, Heinz H. and Combettes, Patrick L.},
	publisher	= {Springer},
	title		= {Convex Analysis and Monotone Operator Theory in {H}ilbert Spaces},
	year		= {2017},
	nodoi		= {10.1007/978-3-319-48311-5},
	isbn		= {978-3-319-48310-8},
	series		= {CMS Books in Mathematics},
}

@article{bento2025convergence,
  title={Convergence of descent optimization algorithms under {P}olyak-{{\L}}ojasiewicz-{K}urdyka conditions},
  author={Bento, Glaydston and Mordukhovich, Boris and Mota, Tiago and Nesterov, Yurii},
  journal={Journal of Optimization Theory and Applications},
  volume={207},
  number={3},
  pages={41},
  year={2025},
  publisher={Springer}
}

@article{bolte2006nonsmooth,
	title		= {A nonsmooth {M}orse--{S}ard theorem for subanalytic functions},
	author		= {Bolte, J\'er\^ome and Daniilidis, Aris and Lewis, Adrian},
	journal		= {Journal of Mathematical Analysis and Applications},
	volume		= {321},
	number		= {2},
	pages		= {729--740},
	year		= {2006},
	publisher	= {Academic Press}
}

@article{bolte2007clarke,
	title		= {Clarke subgradients of stratifiable functions},
	author		= {Bolte, J\'er\^ome and Daniilidis, Aris and Lewis, Adrian and Shiota, Masahiro},
	journal		= {SIAM Journal on Optimization},
	volume		= {18},
	number		= {2},
	pages		= {556--572},
	year		= {2007},
	publisher	= {SIAM}
}

@article{bolte2007lojasiewicz,
	author		= {Bolte, J\'er\^ome and Daniilidis, Aris and Lewis, Adrian},
	nodoi		= {10.1137/050644641},
	issn		= {1052-6234},
	journal		= {SIAM Journal on Optimization},
	number		= {4},
	pages		= {1205--1223},
	title		= {The {{\L}}o\-ja\-sie\-wicz Inequality for Nonsmooth Subanalytic Functions with Applications to Subgradient Dynamical Systems},
	volume		= {17},
	year		= {2007}
}

@article{bolte2014proximal,
	title		= {Proximal Alternating Linearized Minimization for nonconvex and nonsmooth problems},
	author		= {Bolte, J{\'e}r{\^o}me and Sabach, Shoham and Teboulle, Marc},
	year		= {2014},
	issn		= {0025-5610},
	journal		= {Mathematical Programming},
	volume		= {146},
	number		= {1--2},
	nodoi		= {10.1007/s10107-013-0701-9},
	publisher	= {Springer Berlin Heidelberg},
	pages		= {459--494},
	language	= {English},
}

@article{bu2019lqr,
  title={{LQR} through the lens of first order methods: Discrete-time case},
  author={Bu, Jingjing and Mesbahi, Afshin and Fazel, Maryam and Mesbahi, Mehran},
  journal={arXiv preprint arXiv:1907.08921},
  year={2019}
}

@article{chen2014convergence,
  title={The convergence guarantees of a non-convex approach for sparse recovery},
  author={Chen, Laming and Gu, Yuantao},
  journal={IEEE Transactions on Signal Processing},
  volume={62},
  number={15},
  pages={3754--3767},
  year={2014},
  publisher={IEEE}
}

@article{davis2018subgradient,
  title={Subgradient methods for sharp weakly convex functions},
  author={Davis, Damek and Drusvyatskiy, Dmitriy and MacPhee, Kellie J and Paquette, Courtney},
  journal={Journal of Optimization Theory and Applications},
  volume={179},
  pages={962--982},
  year={2018},
  publisher={Springer}
}

@article{fatkhullin2021optimizing,
  title={Optimizing static linear feedback: Gradient method},
  author={Fatkhullin, Ilyas and Polyak, Boris},
  journal={SIAM Journal on Control and Optimization},
  volume={59},
  number={5},
  pages={3887--3911},
  year={2021},
  publisher={SIAM}
}

@article{fatkhullin2022sharp,
  title={Sharp analysis of stochastic optimization under global {K}urdyka-{{\L}}ojasiewicz inequality},
  author={Fatkhullin, Ilyas and Etesami, Jalal and He, Niao and Kiyavash, Negar},
  journal={Advances in Neural Information Processing Systems},
  volume={35},
  pages={15836--15848},
  year={2022}
}

@article{frankel2015splitting,
  title={Splitting methods with variable metric for {K}urdyka--{{\L}}ojasiewicz functions and general convergence rates},
  author={Frankel, Pierre and Garrigos, Guillaume and Peypouquet, Juan},
  journal={Journal of Optimization Theory and Applications},
  volume={165},
  pages={874--900},
  year={2015},
  publisher={Springer}
}

@article{garrigos2023convergence,
  title={Convergence of the forward-backward algorithm: Beyond the worst-case with the help of geometry},
  author={Garrigos, Guillaume and Rosasco, Lorenzo and Villa, Silvia},
  journal={Mathematical Programming},
  volume={198},
  number={1},
  pages={937--996},
  year={2023},
  publisher={Springer}
}

@book{jackie2007deblurring,
  title={Deblurring Images: Matrices, Spectra, and Filtering},
  author={Jackie (Jianhong) Shen},
  year={2007},
  publisher={Society for Industrial and Applied Mathematics}
}

@article{Kabganitechadaptive,
  author = {Kabgani, A. and Ahookhosh, M.},
  title = {{ItsDEAL}: Inexact two-level smoothing descent algorithms for weakly convex optimization},
  year = {2025},
  url = {https://doi.org/10.48550/arXiv.2501.02155},
  journal = {arXiv reprint}
}

@article{Kabganidiff,
  author = {Kabgani, A. and Ahookhosh, M.},
  title = {Moreau envelope and proximal-point methods under the lens of high-order regularization},
  year = {2025},
  volume = {33},
  number = {47},
  journal = {Set-Valued and Variational Analysis }
}

@article{Kabgani24itsopt,
  author = {Kabgani, A. and Ahookhosh, M.},
  title = {{ItsOPT}: An inexact two-level smoothing framework for nonconvex optimization via high-order {M}oreau envelope},
  year = {2024},
  url = {https://doi.org/10.48550/arXiv.2410.19928},
  journal = {arXiv reprint}
}

@article{ahookhosh2025asymptotic,
  title={Asymptotic convergence analysis of high-order proximal-point methods beyond sublinear rates},
  author={Ahookhosh, Masoud and Iusem, Alfredo and Kabgani, Alireza and Lara, Felipe},
  journal={arXiv:2505.20484},
  year={2025}
}

@inproceedings{karimi2016linear,
  author    = {Karimi, Hamed and Nutini, Julie and Schmidt, Mark},
  title     = {Linear {C}onvergence of {G}radient and {P}roximal-{G}radient {M}ethods under the {P}olyak-{{\L}}ojasiewicz {C}ondition},
  editor    = {Frasconi, Paolo and Landwehr, Niels and Manco, Giuseppe and Vreeken, Jilles},
  booktitle = {Machine Learning and Knowledge Discovery in Databases},
  series    = {Lecture Notes in Computer Science},
  volume    = {9851},
  pages     = {795--811},
  publisher = {Springer},
  address   = {Cham},
  year      = {2016}
}

@article{kingma2014adam,
  title={Adam: A method for stochastic optimization},
  author={Kingma, Diederik P},
  journal={arXiv preprint arXiv:1412.6980},
  year={2014}
}

@article{kurdyka1998gradients,
	author		= {Kurdyka, Krzysztof},
	journal		= {Annales de L'institut Fourier},
	language	= {eng},
	number		= {3},
	pages		= {769-783},
	publisher	= {Association des Annales de l'Institut Fourier},
	title		= {On gradients of functions definable in o-minimal structures},
	volume		= {48},
	year		= {1998},
}

@article{li2018calculus,
    title={Calculus of the exponent of {K}urdyka--{{\L}}ojasiewicz inequality and its applications to linear convergence of first-order methods},
    author={Li, Guoyin and Pong, Ting Kei},
    journal={Foundations of Computational Mathematics},
    volume={18},
    number={5},
    pages={1199--1232},
    year={2018},
    publisher={Springer}
}

@article{liu2018robust,
  title={Robust sparse recovery via weakly convex optimization in impulsive noise},
  author={Liu, Qi and Yang, Chengzhu and Gu, Yuantao and So, Hing Cheung},
  journal={Signal Processing},
  volume={152},
  pages={84--89},
  year={2018},
  publisher={Elsevier}
}

@article{lojasiewicz1963propriete,
	title		= {Une propri{\'e}t{\'e} topologique des sous-ensembles analytiques r{\'e}els},
	author		= {{\L}ojasiewicz, Stanislaw},
	journal		= {Les {\'E}quations Aux D{\'e}riv{\'e}es Partielles},
	pages		= {87--89},
	year		= {1963}
}

@article{lojasiewicz1993geometrie,
	author		= {{\L}ojasiewicz, Stanislaw},
	journal		= {Annales de L'institut Fourier},
	language	= {fre},
	number		= {5},
	pages		= {1575-1595},
	publisher	= {Association des Annales de l'Institut Fourier},
	title		= {Sur la g{\'e}om{\'e}trie semi- et sous- analytique},
	volume		= {43},
	year		= {1993},
}

@InProceedings{mei2020global,
  title = 	 {On the {G}lobal {C}onvergence {R}ates of {S}oftmax {P}olicy {G}radient {M}ethods},
  author =       {Mei, Jincheng and Xiao, Chenjun and Szepesvari, Csaba and Schuurmans, Dale},
  booktitle = 	 {Proceedings of the 37th International Conference on Machine Learning},
  pages = 	 {6820--6829},
  year = 	 {2020},
  editor = 	 {III, Hal Daumé and Singh, Aarti},
  volume = 	 {119},
  series = 	 {Proceedings of Machine Learning Research},
  publisher =    {PMLR}
}

@book{Mordukhovich2018,
  title={Variational Analysis and Applications},
  author={Mordukhovich, Boris Sholimovich},
  volume={},
  year={2018},
  publisher={Springer Cham}
}

@article{moreau1965proximite,
	author		= {Moreau, Jean-Jacques},
	title		= {Proximité et dualité dans un espace {H}ilbertien},
	journal		= {Bulletin de la Société Mathématique de France},
	language	= {fre},
	pages		= {273-299},
	publisher	= {Société mathématique de France},
	volume		= {93},
	year		= {1965},
}

@article{nesterov2015universal,
  title={Universal gradient methods for convex optimization problems},
  author={Nesterov, Yurii},
  journal={Mathematical Programming},
  volume={152},
  number={1-2},
  pages={381--404},
  year={2015},
  publisher={Springer}
}

@book{nesterov2018lectures,
  title={Lectures on Convex Optimization},
  author={Nesterov, Yurii},
  volume={137},
  year={2018},
  publisher={Springer}
}

@book{nocedal2006numerical,
	Author		= {Nocedal, Jorge and Wright, Stephen},
	publisher	= {Springer Science \& Business Media},
	Title		= {Numerical {O}ptimization},
	Year		= {2006},
}

@article{nurminskii1973quasigradient,
  title={The quasigradient method for the solving of the nonlinear programming problems},
  author={Nurminskii, Evgeni Alekseevich},
  journal={Cybernetics},
  volume={9},
  pages={145--150},
  year={1973},
  publisher={Springer}
}

@article{poliquin1996proxregular,
	issn		= {00029947},
	author		= {Poliquin, Ren\'e A. and Rockafellar, R. Tyrrell},
	journal		= {Transactions of the American Mathematical Society},
	number		= {5},
	pages		= {1805-1838},
	publisher	= {American Mathematical Society},
	title		= {Prox-Regular Functions in Variational Analysis},
	volume		= {348},
	year		= {1996}
}

@article{rahimi2024projected,
  title={Projected subgradient methods for paraconvex optimization: Application to robust low-rank matrix recovery},
  author={Rahimi, Morteza and Ghaderi, Susan and Moreau, Yves and Ahookhosh, Masoud},
  url ={https://doi.org/10.48550/arXiv.2501.00427},
  journal={arXiv preprint},
  year={2024}
}

@book{rockafellar2011variational,
	Author		= {Rockafellar, R. Tyrrell and Wets, Roger J.-B.},
	publisher	= {Springer Science \& Business Media},
	Title		= {Variational {A}nalysis},
	Volume		= {317},
	Year		= {2011},
}

@article{rodomanov2020smoothness,
  title={Smoothness parameter of power of {E}uclidean norm},
  author={Rodomanov, Anton and Nesterov, Yurii},
  journal={Journal of Optimization Theory and Applications},
  volume={185},
  pages={303--326},
  year={2020},
  publisher={Springer}
}

@article{stella2017forward,
	author		= {Stella, Lorenzo and Themelis, Andreas and Patrinos, Panagiotis},
	title		= {Forward-backward quasi-{N}ewton methods for nonsmooth optimization problems},
	journal		= {Computational Optimization and Applications},
	year		= {2017},
	month		= {Jul},
	day			= {01},
	volume		= {67},
	number		= {3},
	pages		= {443--487},
	issn		= {1573-2894},
	nodoi		= {10.1007/s10589-017-9912-y},
}

@article{themelis2018forward,
	author		= {Themelis, Andreas and Stella, Lorenzo and Patrinos, Panagiotis},
	title		= {Forward-Backward Envelope for the Sum of Two Nonconvex Functions: Further Properties and Nonmonotone Linesearch Algorithms},
	journal		= {SIAM Journal on Optimization},
	volume		= {28},
	number		= {3},
	pages		= {2274-2303},
	year		= {2018},
	nodoi		= {10.1137/16M1080240},
}

@article{yashtini2016global,
 title={On the global convergence rate of the gradient descent method for functions with {H}{\"o}der continuous gradients},
  author={Yashtini, Maryam},
  journal={Optimization Letters},
  volume={10},
  pages={1361--1370},
  year={2016},
  publisher={Springer}
}

@article{yu2022kurdyka,
  title={Kurdyka--{{\L}}ojasiewicz exponent via inf-projection},
  author={Yu, Peiran and Li, Guoyin and Pong, Ting Kei},
  journal={Foundations of Computational Mathematics},
  volume={22},
  number={4},
  pages={1171--1217},
  year={2022},
  publisher={Springer}
}

@InProceedings{yuan2022general,
  title = 	 { A {G}eneral {S}ample {C}omplexity {A}nalysis of {V}anilla {P}olicy {G}radient },
  author =       {Yuan, Rui and Gower, Robert M. and Lazaric, Alessandro},
  booktitle = 	 {Proceedings of the 25th International Conference on Artificial Intelligence and Statistics},
  pages = 	 {3332--3380},
  year = 	 {2022},
  editor = 	 {Camps-Valls, Gustau and Ruiz, Francisco J. R. and Valera, Isabel},
  volume = 	 {151},
  publisher =    {PMLR}
}

@article{zeng2018global,
  title={Global convergence in deep learning with variable splitting via the {K}urdyka-{{\L}}ojasiewicz property},
  author={Zeng, Jinshan and Ouyang, Shikang and Lau, Tim Tsz-Kit and Lin, Shaobo and Yao, Yuan},
  journal={arXiv:1803.00225},
  year={2018}
}

@article{yang2019weakly,
  title={Weakly convex regularized robust sparse recovery methods with theoretical guarantees},
  author={Yang, Chengzhu and Shen, Xinyue and Ma, Hongbing and Chen, Badong and Gu, Yuantao and So, Hing Cheung},
  journal={IEEE Transactions on Signal Processing},
  volume={67},
  number={19},
  pages={5046--5061},
  year={2019},
  publisher={IEEE}
}

\end{document}